\documentclass{article}

\usepackage[english]{babel}

\usepackage[letterpaper,left=30mm,right=30mm,top=22mm,bottom=22mm,marginparwidth=1.5cm]{geometry} %

\usepackage{bbm}
\usepackage{amsmath, amssymb, mathtools}
\usepackage{graphicx}
\usepackage{amsthm}
\usepackage[colorlinks=true, allcolors=blue]{hyperref}
 \usepackage{authblk}

\usepackage[utf8]{inputenc}

\newtheorem{theorem}{Theorem}[section]
\newtheorem{lemma}[theorem]{Lemma}
\newtheorem{prop}[theorem]{Proposition}
\newtheorem{defn}[theorem]{Definition}
\newtheorem{remark}[theorem]{Remark}

\newcommand{\norm}[1]{\left\lVert #1 \right \rVert}
\newcommand{\f}{\boldsymbol}

\title{On reduced inertial PDE models for Cucker-Smale flocking dynamics}

\author[a,c,$\star$]{Sebastian Zimper}
\author[b,$\star$]{Federico Cornalba}
\author[a]{Nata\v sa Djurdjevac Conrad}
\author[c]{Ana Djurdjevac}

\affil[a]{Zuse Institute Berlin, Germany}
\affil[b]{Department of Mathematical Sciences, University of Bath, UK}
\affil[c]{Institut f\"ur Mathematik und Informatik, Freie Universit\"at Berlin, Berlin, Germany}
\affil[$\star$] {S. Zimper and F. Cornalba have contributed equally to this work.}

\begin{document}

\maketitle
\begin{abstract}
  In particle systems, flocking refers to the phenomenon where particles’ individual velocities eventually align. The Cucker-Smale model is a well-known mathematical framework that describes this behavior. Many continuous descriptions of the Cucker-Smale model use PDEs with both particle position and velocity as independent variables, thus providing a full description of the particles mean-field limit (MFL) dynamics. In this paper, we introduce a novel reduced inertial PDE model consisting of two equations that depend solely on particle position. 
In contrast to other reduced models, ours is not derived from the MFL, but directly includes the model reduction at the level of the empirical densities, thus allowing for a straightforward connection to the underlying particle dynamics.
We present a thorough analytical investigation of our reduced model, showing that: firstly, our reduced PDE satisfies a natural and interpretable continuous definition of flocking; secondly, in specific cases, we can fully quantify the discrepancy between PDE solution and particle system. 
Our theoretical results are supported by numerical simulations.   \\
  


\end{abstract}

\section{Introduction}
\label{sec:Intro}

%\sebastian{When going through the document I have found some of the following reuses of variables: $j$, the empirical momentum density and an index used in sums; $B$, the Brownian motion used in \eqref{eq:stoc_det_v}, also a constant in the cutoff function \eqref{cutoff} and re-used in \eqref{eq:Split_bound}; $v$ as an arbitrary $H^{-1}$ function, but could also be confused with the velocity; $K$ \eqref{ass_boundedness_rho} and as a constant for the cutoff \eqref{cutoff}; $A$ used as a constant in Lemma \ref{well_posedness_reg} and in the beggining of Section \ref{sec:ReductionAccuracy} as well as \eqref{eq:Split_bound}; $f$ and $g$ are used as generic $H^{-1}$ functions in Section \ref{sec:ReductionAccuracy}, but are also the distribution of the Mean-Field pde and function in the general interaction function, respectively. }

Collective motion is used to describe a wide variety of phenomena observed in nature, in which large numbers of interacting particles self-organise from an unordered to an ordered state. The particles can represent, for example, bacteria, fish, birds and other vertebrates found in large groups, all of which exhibit collective motion due to the interactions between individuals (see, for example, the review \cite{Couzin2003}). One of the most common cases of collective motion is flocking, the alignment of the velocities of all particles. Mathematical models to describe flocking are usually expressed as inertial interacting particle systems, where the interaction between particles leads to a velocity alignment (see \cite{Vicsek2012} and the references therein). One of the prototypical models of this sort is the Cucker-Smale (CS) model \cite{cucker2007emergent,Cucker2007b}, describing a deterministic system where the influence of one particle on another is weighed by some interaction potential (also referred to as the communication rate). There exist various analytical results on the conditions necessary for flocking and on the rate of flocking, considering for example the initial positions and velocities of particles, as well as system parameters \cite{cucker2007emergent, Cucker2007b,HaTadmor2008, HaLiu2009, Ahn2012}.

%As the number of particles in these systems grows large \ana{I think this can be misleading, because we will consider $N$ fixed, maybe: if the number of particles in these systems is very big}, it becomes unfeasible \ana{difficult? - some people do manage with very big number of particles} to simulate their dynamics. 
% For large systems with many particles, simulating their dynamics becomes computationally expensivce.
Simulating the dynamics of such systems when the number of particles is very large becomes computationally expensive. To overcome this challenge, the mean-field limit (MFL) is used to give a macroscopic description of the system capturing the overall distribution of the particles' positions and velocities. For the CS model, rigorous derivations and analysis of flocking for the mean-field limit are presented in \cite{HaTadmor2008,HaLiu2009}. The simulation of the MFL can however still be computationally expensive, as it is necessary to solve a high dimensional partial differential equation (PDE). 
This high computational cost has motivated the derivation of several reduced models \cite{CarrilloEtAl2010A,CarrilloEtAl2010B,Tadmor2014, Choi2017} which -- in essence -- trade off some accuracy in exchange for a PDE living on a smaller space, while still exhibiting flocking. More specifically, these reduced models are usually obtained following two basic principles: 
i) key quantities in these reduced models -- such as local mass, momentum and energy densities -- are derived from the MFL itself; 
ii) the hydrodynamic equations of such leading quantities are not closed due to the model reduction, but can be closed in relevant microscopic regimes (the most notable of which is the so-called \emph{mono-kinetic ansatz} \cite{CarrilloEtAl2010B}), thus giving a self-contained, computationally efficient  continuous model which still captures flocking.

In this work, we propose and analyze an alternative reduced PDE for the CS flocking dynamics which, crucially, is \emph{not} derived from the MFL. Recall that in the MFL derivation the starting point is the empirical measure that involves the Dirac delta of position and velocity. On the other hand, our model is built by directly defining  the empirical particle density $\rho$ and momentum density $\f{j}$ which are functions of the space variable only \eqref{rho_j}, and that feature the velocity field in a multiplicative manner, similar to a local Eulerian velocity definition. In the reduced model, flocking velocity will appear as a (random) parameter. %} \ana{discuss previous sentence - previous two sentences should explain better the following one - I would erase the Eulerican part} \sebastian{I like the change, but would remove the last sentence as it is not clear to me what the velocity field being a random parameter means.} \ana{The last comment should refer the common question that I got few tames, that for us now $\overline{v}$ is parameter, and it is not just parameter because it is random since IC are random, so it should refer to that}
%Specifically, our model is built by directly defining empirical particle density $\rho$ and momentum density $\f{j}$ which are functions of the space variable only, and which feature the velocity field in a multiplicative way reminiscent of a local Eulerian velocity definition.
Crucially, our reduced model relies solely on a closing approximation for the particle system, where the approximation is tied to the velocity alignment. 
The main motivation for having a reduced model which bypasses the MFL entirely, is %mostly related to the modelling choice of singling out the velocity contributions directly at the level of the empirical densities. 
to directly model the velocity contributions at the level of the empirical densities. This has numerous advantages, which include: a) a better control of the closing approximation resulting from velocity alignment; b) a simple and intuitive algebraic condition identifying flocking at the PDE level; c) a convenient analytical setting for framing the quantitative analysis of fluctuations in the velocity field. Although we do not consider these fluctuations in this paper, we plan to study them in subsequent works, and see this work as laying the necessary modelling foundations for such an effort. On a similar note, it is worth mentioning that the derivation of our reduced model is largely compatible with the analytical framework of reduced inertial models of fluctuating hydrodynamics 
\cite{cornalba2021well,cornalba2020weakly,Cornalba2019}, thus enabling exchange of methodologies and providing additional foundations for future works on fluctuations. 
%Additionally, we are aware also of \cite{ha2018first}, which derives a reduced first-order Cucker-Smale model without relying on the MFL, but focuses entirely on the microscopic (particle) description.
%\natasa{@Federico: I am not sure I understand what this sentence says. Do we really need it? I am sure we dont mention also some other methods that are not based on MFL, but well...I would suggest to delete this sentence.} 

Our proposed reduced model expands the scope of existing approaches to model  flocking. %In particular, we have derived our reduced PDE model to accurately capture the behaviour of the CS model close to the flocking regime \sebastian{Is the flocking regime when flocking has occured or when we are sufficiently close to flocking? If it is the later I would rewrite this sentence and the next as: 
In particular, we have derived our reduced PDE model to accurately capture the behaviour of the CS model in the flocking regime, the state where the particles' velocities are sufficiently aligned.
% have aligned their velocities sufficiently. In this state, particles of a flock reach consensus on a particular velocity alignment. 
Understanding the dynamics, particularly in the flocking regime is crucial, as it can reveal challenges associated with introducing external influences that effectively change and steer the flock. Our main motivation comes from potential applications to systems where influencing flocking patterns is of particular relevance, such as the biological examples of birds and fish schools discussed above.  Similar challenges arise in other systems, such as in online social platforms, where developing influencer and media strategies to influence overall opinion distribution can be difficult when strong alignment exists \cite{Helfmann2023}. In this context, we view our reduced model as a basis for developing computationally efficient tools to study and potentially steer such collective motion in various systems.

We now informally explain the contents of our main results (Theorem \ref{thm:all_PDE_flocking_results} and Theorem \ref{thm_error_PDE_particles_H_minus_2}). %In order to study particle systems in the flocking regime, we derive an inertial reduced PDE model using ideas from \cite{cornalba2021well,cornalba2020weakly,Cornalba2019} \sebastian{the first part of this sentence is a repition, I would remove it}, and 
First, we define an associated, intuitive notion of flocking for our reduced PDE model (this notion is partially reminiscent of the one presented in \cite{Karper2013}, even though its derivation does not rely on the MFL).
We then prove analytically that for a specific class of interaction potentials, the PDE model exhibits flocking (Theorem \ref{thm:all_PDE_flocking_results}). Furthermore, in dimension one, we quantify the discrepancy between our reduced model and the dynamics of the particle system (Theorem \ref{thm_error_PDE_particles_H_minus_2}). 
%We show that the discrepancy between the systems is dependent on the following factors \federico{these factors might change once we switch from $H^{-1}$ to $H^{-2}$}: the total number of particles; the rate at which the particle and PDE systems flock; and the accuracy with which we identify the positions of the particles via mollification. 
We present Fig.~\ref{fig:Schematic} as a schematic illustration of our approach to derive the reduced PDE model from the particle model.    
To further improve the accuracy of our reduced PDE model,  %in reproducing the underlying particle system,
particularly  during the pre-flocking phase, we introduce a modification of the PDE which greatly decreases the discrepancy between the PDE and the underlying particle model without increasing the  computational effort. We prove the well-posedness of this modification as well as that it preserves PDE flocking, and demonstrate its positive impact using extensive numerical results. Analytical results quantifying the improvement in accuracy of this modification are, however, still lacking. We also demonstrate that simulating the PDE is significantly faster than the CS model for larger numbers of particles, justifying the use of the reduced model. To compare our reduced inertial model and the hydrodynamic equations, we use numerical simulations and show that the two models are not identical, and that the two reduced models have a similar level of accuracy in reproducing the results of the particle system. Lastly, we observe that if noise is included in the particle dynamics (such as in \cite{10.1214/18-AAP1400}), a reduced stochastic PDE can be formally derived which we observe to also accurately approximate the particle system in the flocking regime. Further analytical considerations of the associated analysis of fluctuations will be studied in the future.  

The rest of this paper is organised as follows. The Cucker-Smale model along with the derivation of our reduced PDE model are discussed in Section \ref{sec:ModelReduction}. Our two main results on PDE flocking, and discrepancy between PDE and CS model are stated in Section \ref{sec:InformalMain}. Section \ref{sec:PDEflock} is devoted to proving all necessary analytical results substantiating the first main result on PDE flocking, while Section \ref{sec:ReductionAccuracy} is devoted to proving the second main result on the accuracy of the reduced model in approximating the particle system. Various numerical results are presented in Section \ref{sec:Numerical}, %for a more general interaction potential for which we lack analytical results 
including the experiments on the accuracy of the reduced model, a comparison with the hydrodynamic model, the computational effort of the particle system and reduced model, and the inclusion of noise in the dynamics. Finally, in Section \ref{sec:Conclusion} we present our conclusions and highlight some open questions. In the appendix we provide some of the technical details of the proofs from Sections \ref{sec:PDEflock} and \ref{sec:ReductionAccuracy}. 

\begin{figure}[h]
\centering
    \includegraphics[width=0.75\linewidth]{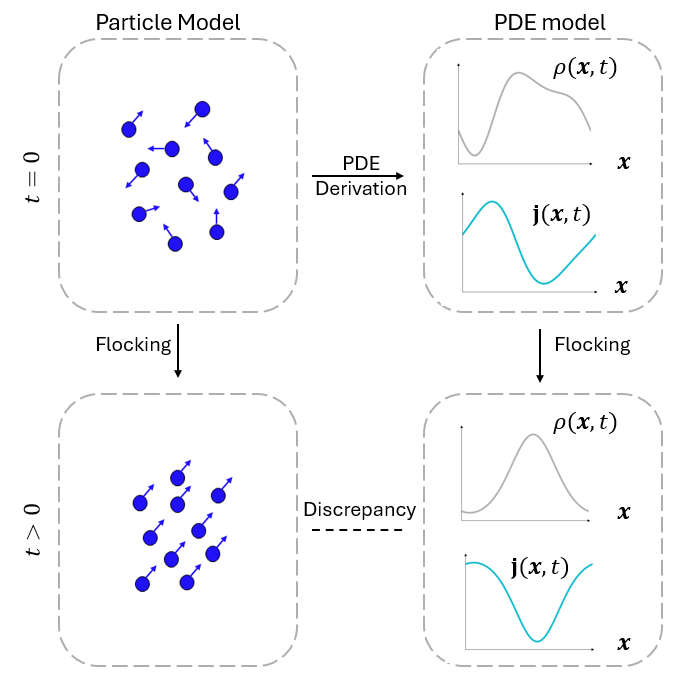}
\caption{Schematic overview of our approach for deriving the reduced PDE model from a particle model. Blue circles represent particles characterised by their position (center of the circle) and velocity (indicated by the arrows). In the flocking regime, particles align their velocities. PDE model is given by the empirical density $\rho$ and momentum density $\f{j}$.}  
\label{fig:Schematic}
\end{figure}

\section{Model reduction PDE vs CS}
\label{sec:ModelReduction}

\subsection{Cucker-Smale model}
\label{subsec:CSmodel}

We begin by introducing the well-established CS model \cite{cucker2007emergent,Cucker2007b} of flocking dynamics, that describes how particles interact with each other in order to align their movements. We consider the basic case associated with the purely deterministic second-order dynamics for $N$ particles moving in a $d-$ dimensional space\footnote{In the original work this space was $\mathbb{R}^d$. Here, we consider the torus $\mathbb{T}^d$ as it simplifies our analysis.}, where $N\geq 2$ and $d \geq 1$. The motion of each particle $i$ is represented by its position in a $d-$dimensional torus $\f{x}_i\in \mathbb{T}^d$ and its velocity $\f{v}_i \in \mathbb{R}^d$. The evolution of the Cucker-Smale model is governed by the system
\begin{align}
    \dot{\f{x}}_i (t) & = \f{v}_i (t), \label{particles_det_x} \\
    \dot{\f{v}}_i (t) & = N^{-1}\sum_{j=1}^{N}{a(\f{x}_i (t)-\f{x}_j (t))(\f{v}_j (t)-\f{v}_i (t))},\label{particles_det_v}
\end{align}
for all $i=1,\dots,N$, and with (possibly random) initial conditions 
$\{(\f{x}_i (0), \f{v}_i (0))\}_{i=1}^{N}$. 
The non-negative function $a$ denotes a pairwise interaction potential that determines the strength of the influence that particles have on each other. In the standard Cucker-Smale model $a$ is a function of a distance between particles and is thus symmetric, such that particles $i$ and $j$ have the same influence on the alignment of each other. The symmetry property of $a$ is the key ingredient to derive the flocking behavior \cite{HaTadmor2008}. Extensions to asymmetric models for studying flocking, e.g. with the presence of "leaders" have also been considered \cite{Motsch2011}. Throughout this work, we use one of the earliest and most standard choices for $a$ \cite{cucker2007emergent}, namely
\begin{equation}
    \label{eq:NonConstComm}
    a(\f{x}_i(t) - \f{x}_j(t)) = \frac{\lambda}{(1 + \|\f{x}_i(t) - \f{x}_j(t)\|^2 )^r },
\end{equation}
for some fixed parameters $r, \lambda \geq 0$, which determine the decay and coupling strength of the interaction potential, respectively. %The flocking regime of the Cucker-Smale model assumes  that all the particles' velocities $\f{v}_i(t)$ converge to the initial mean velocity
The CS model is said to exhibit flocking if all the particles' velocities $\f{v}_i(t)$ converge to the initial mean velocity 
\begin{align}\label{particle_flocking_velocity}
    \overline{\f{v}} := N^{-1}\sum_{i=1}^N \f{v}_i(0),
\end{align}
while the distance between the particles remains 
(uniformly) bounded in time. It is easy to see that the mean velocity of the system $\overline{\f{v}}(t) = N^{-1} \sum_{i=1}^N \f{v}_i(t)$ is conserved under the dynamics \eqref{particles_det_x}--\eqref{particles_det_v} as $a(\cdot)$ is an even function \cite{10.1214/18-AAP1400}. We recall the formal definition of flocking \cite{cucker2007emergent,Cucker2007b}. 
\begin{defn}[Particle flocking in $L^{p},L^{q}$]\label{def:Particleflocking}
An interacting particle system $\{(\f{x}_i,\f{v}_i)\}_{i=1}^N$ is said to flock in $L^{p},L^{q}$ if the following two conditions are satisfied:
\begin{itemize}
    \item \textbf{Velocity alignment:}  \begin{align}\label{particle_vel_alignment}
    \lim_{t\rightarrow+\infty} \mathbb{E} \left[ \|\f{v}_i (t) - \overline{\f{v}}  \|^p \right]= 0, \quad 1\leq i \leq N.
    \end{align}
    \item \textbf{Group formation\footnote{Since we consider the space $\mathbb{T}^d$ this condition will always be fulfilled.}:} 
\begin{align}\label{particle_pos_alignment}
    \sup_{0\leq t<\infty} \mathbb{E}[\|\f{x}_i (t) - \f{x}_j (t) \|^q] < \infty, \quad 1\leq i, j \leq N.
    \end{align}
\end{itemize}
\end{defn}
A variety of results exist on the system parameters for which the Cucker-Smale model flocks unconditionally, i.e., for all initial distributions of particles, as well as conditionally. In particular, for the interaction potential \eqref{eq:NonConstComm}, unconditional flocking occurs if $r \leq 1/2$, see \cite{cucker2007emergent,Cucker2007b, HaTadmor2008, HaLiu2009,CarrilloEtAl2010A}. We consider only unconditional flocking in this work and thus set $r=1/2$ in numerical simulations.

For our purposes, it will be convenient to stick to a declination of flocking which is closely tied to that of Definition \ref{def:Particleflocking}. Specifically, following the lines of \cite{HaLiu2009}, assuming the initial support of the velocities to be compactly supported in a sphere of radius $R_v$ (i.e., $supp(v_0)\subset B_{R_v}(0)$), there is particle flocking for the Cucker-Smale system \eqref{particles_det_x}--\eqref{particles_det_v} in the sense that there
exists an $x_M$ such that 
\begin{align}\label{particle_flocking_bound}
\left(\sum_{i=1}^{N}{|\f{v}_i - \overline{\f{v}}|^2}\right)^{1/2} \leq \left(\sum_{i=1}^{N}{|\f{v}_i(0) - \overline{\f{v}}|^2}\right)^{1/2} e^{-a(x_M)t} \leq N^{1/2}R_v e^{-a(x_M)t}.
\end{align}

\begin{remark}\label{shifting_vs}
Since shifting $\f{v}_i$ by a constant only impacts the dynamics of the system by a translation of the particle position's, we always consider $\overline{\f{v}}$ with $\overline{v}_i\neq 0, i=1,\dots,d$, in the upcoming sections. This will simplify some technical arguments.    
\end{remark}

To study the Cucker-Smale model when the number of particles is large, it is useful to consider the MFL of the system, describing the evolution of the distribution of particles $f(\f{x},\f{v},t)$ at position $\f{x}$ and with velocity $\f{v}$ at time $t$. The mean-field equation is given by
\begin{align}
    \label{eq:MFLsystem}
    \partial_t f + \f{v} \cdot	\nabla_{\f{x}} f + \nabla_{\f{v}} \cdot Q(f,f) = 0 ,
\end{align}
where 
\begin{align*}
    Q(f,f)(\f{x},\f{v},t) := \int a( \f{x} - \f{y}) (\f{v}' - \f{v}) f(\f{x},\f{v},t) f(\f{y},\f{v}',t) d\f{v}' d \f{y} .
\end{align*}
This kinetic description for the Cucker-Smale model was derived in \cite[Section $3$]{HaTadmor2008}, where the global existence of smooth solutions was proved \cite[Theorem $3.1$]{HaTadmor2008}. Furthermore, by analysing the energy fluctuations around the mean velocity it was shown that for $r < 1/4$ the kinetic model also exhibits flocking \cite[Theorem $4.1$]{HaTadmor2008}. In \cite[Theorem 6.2]{HaLiu2009} the global existence of measure valued solutions of 
\eqref{eq:MFLsystem} was proved. 

A reduced PDE, or hydrodynamic model, is formally obtained in \cite{HaTadmor2008} by taking moments of the solution to \eqref{eq:MFLsystem}. The macroscopic quantities of interest are: the mass $\rho := \int f d\f{v} $; the momentum $\rho \f{u} := \int \f{v} f  d\f{v}$, where $\f{u}(\f{x},t)$ is referred to as the mean velocity field; and the energy $\rho E := \int |\f{v}|^2 f d\f{v}$. The resulting system \cite[Equation 5.3]{HaTadmor2008} is however not closed with respect to these quantities. In \cite{CarrilloEtAl2010B}, a closed system is obtained by assuming that fluctuations around the mean velocity field are negligible, and considering the mono-kinetic ansatz, i.e., $f(\f{x},\f{v},t) = \rho(\f{x},t)  \delta(\f{v} - \f{u}(\f{x},t) )$. The resulting, closed, hydrodynamic model is 
\begin{align}
    \partial_t \rho & = - \nabla \cdot (\rho \f{u}) , \label{eq:Hydrodynamic1} \\
    \partial_t (\rho u_i) & = \rho \cdot \left( a * ( \rho u_i) \right) - (\rho u_i) \cdot \left( a *  \rho \right)  - \nabla \cdot (\rho \f{u} u_i ) , \label{eq:Hydrodynamic2}
\end{align}
where $u_i$ is the $i$-th component of $\f{u}$.

The existence of global, classical solutions of \eqref{eq:Hydrodynamic1}-\eqref{eq:Hydrodynamic2} on the periodic domain was proved in \cite[Theorem 3.1]{HaKang2014} under suitable assumptions on the interaction potential and initial data. Additionally, it is shown that \eqref{eq:Hydrodynamic1}-\eqref{eq:Hydrodynamic2} exhibits flocking, using a Lyapunov functional approach measuring the total velocity fluctuation around the mean velocity \cite[Lemma 2.2]{HaKang2014}. This result relies on the interaction potential having a positive lower bound due to the compactness of the domain, which we will also use to prove flocking for our reduced PDE model in Section \ref{sec:PDEflock}. In \cite{HaKang2015} this study is extended to a moving domain. The hydrodynamic limit has also been derived and studied for various extensions of the Cucker-Smale model, for example: when the particles are interacting with a viscous incompressible fluid \cite{HaKang2014}; for non-symmetric interaction potentials \cite{Motsch2011}; and when there is noise present in the particle dynamics \cite{Karper2015}. To the best of our knowledge the accuracy with which the reduced, hydrodynamic model reproduces the dynamics of the particle system has however not been investigated.

%\natasa{TODO: make the terminology consistent: intercation function/kernel/potential, communication rate; do we need to write $\f{x}_j (t)$ or $\f{x}_j$ is ok?}

\subsection{Derivation of reduced PDE model for flocking}
\label{subsec:PDEderiv}

In order to derive a continuous reduced model for \eqref{particles_det_x}--\eqref{particles_det_v}, we characterise the dynamics of the empirical density of particles $\rho$ and the empirical momentum density of particles $\f{j}$ given by
\begin{align}
    \rho(\f{x},t) & := N^{-1}\sum_{i=1}^{N}{\delta(\f{x}-\f{x}_i(t))} \label{rho_j} ,\qquad 
    \f{j}(\f{x},t) := N^{-1}\sum_{i=1}^{N}{\f{v}_i(t)\delta(\f{x}-\f{x}_i(t))},
\end{align}
 where $\delta$ denotes the Dirac distribution. 
 As already mentioned in the introduction, and in contrast to the MFL approach described in ther previous section, the densities \eqref{rho_j} are structurally associated with a reduced model, as the atomic measures are evaluated in time and position only, and the velocity enters just as the multiplicative factor.
 
 A formal application of the chain rule on \eqref{rho_j} gives us
\begin{align*}
    \partial_t \rho(\f{x},t) & = -N^{-1}\sum_{i=1}^{N} \sum_{z=1}^d \partial_{x_z} \delta(\f{x}-\f{x}_i(t))\f{v}_{i,z} (t) = - \nabla \cdot \f{j} ,
\end{align*}
and for the $m$-th component of $\f{j}$ ,
 \begin{align}\label{Evolution_J}
     \partial_t j_m(x,t) & = N^{-1}\sum_{i=1}^{N}{\partial_t {v}_{i,m}(t)\delta(\f{x}-\f{x}_i(t))} + N^{-1}\sum_{i=1}^{N}{v_{i,m}(t)\partial_t{\delta}(\f{x}-\f{x}_i(t))}  \nonumber\\
     & = N^{-1}\sum_{i=1}^{N} \left(N^{-1}\sum_{j=1}^{N}{a(\f{x}_i(t)-\f{x}_j(t))(v_{j,m}(t)-v_{i,m}(t))}\right) \delta(\f{x}-\f{x}_i(t))  \nonumber\\
     & \quad\quad - N^{-1}\sum_{i=1}^{N}{v_{i,m}(t) \sum_{z=1}^d \partial_{x_z} \delta(\f{x}-\f{x}_i(t)) v_{i,z}(t) } =: T_1 + T_2.
 \end{align}

Splitting the terms involving $\f{v}_i$ and $\f{v}_j$, we carry on in \eqref{Evolution_J} and obtain
\begin{align}\label{derive_convolutions_atomic}
T_1 & = N^{-1}\sum_{i=1}^{N}{\delta(\f{x}-\f{x}_i(t)) N^{-1}\sum_{j=1}^{N}{a(\f{x}_i(t)-\f{x}_j(t))v_{j,m}(t)}} \nonumber\\
& \quad\quad - N^{-1}\sum_{i=1}^{N}{v_{i,m}(t) \delta(\f{x}-\f{x}_i(t)) N^{-1}\sum_{j=1}^{N}{a(\f{x}_i(t)-\f{x}_j(t))}} \nonumber\\
& = N^{-1}\sum_{i=1}^{N}{\delta(\f{x}-\f{x}_i(t)) N^{-1}\sum_{j=1}^{N}{a(\f{x}-\f{x}_j(t))v_{j,m}(t)}}\nonumber \\
& \quad\quad - N^{-1}\sum_{i=1}^{N}{v_{i,m}(t) \delta(\f{x}-\f{x}_i(t)) N^{-1}\sum_{j=1}^{N}{a(\f{x}-\f{x}_j(t))}},
\end{align}
where in the second equality we have replaced $a(\f{x}_i(t)-\f{x}_j(t))$ with $a(\f{x}-\f{x}_j(t))$, since this is formally justified by the Dirac delta function $\delta(\f{x}-\f{x}_i(t))$ providing a non-zero contribution only when $\f{x}=\f{x}_i(t)$. This last argument is reminiscent of the derivation given in \cite{Dean1996}. 

Using the definition of the densities $\rho$ and $\f{j}$ in \eqref{rho_j}, as well as re-writing the sums over the index $j$ as convolutions in space in \eqref{derive_convolutions_atomic}, we obtain
\begin{align}
    \partial_t j_m(\f{x},t) & = \rho (a\ast j_m) - j_m (a \ast \rho) + T_2.
\end{align}
The term $T_2$ needs closing with respect to the densities $\rho, \f{j}$ (for more extensive discussions about this, we refer to \cite{cornalba2021well}). What is important to note here is that, assuming the flocking has occurred (i.e.,  $\f{v}_j\equiv \overline{\f{v}}$ for all $j= 1,\dots,N$), then the term $T_2$ can be re-written by: i) pulling out the common flocking velocity $\overline{\f{v}}$, and ii) identifying differential operators in $\rho,\f{j}$. 
Such an approximation is not unique, and gives us some modelling freedom, as we now detail.

\subsubsection{Reduced PDE model in dimension $d=1$}
\label{subsub:1d}
Here, we approximate $T_2 \approx - \overline{v}^2 \nabla \rho$, thus obtaining the reduced PDE model
\begin{align}
    \partial_t \rho(x,t) & = - \nabla \cdot j \label{reduced_CS_rho} ,\\
    \partial_t j(x,t) & = \left\{\rho (a\ast j)  - j (a \ast \rho) - \overline{v}^2 \nabla \rho\right\} \label{reduced_CS_j}.
\end{align}

\subsubsection{Enhanced reduced PDE model in dimension $d=1$: space-dependent weights}
\label{subsub:1dweighted}
In order to improve the approximation $T_2 \approx - \overline{v}^2 \nabla \rho$ outside of the flocking regime, we allow for a weight $w=w(x,t)$ and consider the ``weighted'' 1d PDE model
\begin{align}
    \partial_t \rho(x,t) & = - \nabla \cdot j \label{reduced_CS_rho_w} ,\\
    \partial_t j(x,t) & = \left\{\rho (a\ast j)  - j (a \ast \rho) - \overline{v}^2 \nabla (w\rho)\right\} \label{reduced_CS_j_w},
\end{align}
as thoroughly discussed in Subsection \ref{subsec:Weights}. 

\subsubsection{Reduced PDE model in dimension $d>1$}
\label{subsub:d_gr_1}
We approximate $T_2$ as
$
    [T_2]_m \approx -N^{-1} \sum_{i=1}^N \overline{v}_m \sum_{z=1}^d \partial_{x_z} \delta(\f{x}-\f{x}_i(t)) v_{i,z} = - \overline{v}_m \nabla \cdot \f{j}.
    %& = - \overline{v}_m \sum_{z=1}^d \partial_{x_z} \left( N^{-1} \sum_{i=1}^N v_{i,z} \delta(\f{x}-\f{x}_i(t))  \right) = - \overline{v}_m \nabla \cdot \f{j} 
$
and consequently end up with the reduced PDE model
\begin{align}
    \partial_t \rho(\f{x},t) & = - \nabla \cdot \f{j}  \label{reduced_CS_rho_d},\\
    \partial_t \f{j}(\f{x},t) & = \left\{\rho (a\ast \f{j})  - \f{j} (a \ast \rho) - \overline{\f{v}} \nabla \cdot \f{j}\right\}\label{reduced_CS_j_d} .
\end{align}

\begin{remark}
    While the reduced model \eqref{reduced_CS_rho_d}--\eqref{reduced_CS_j_d} could -- in principle -- also be used in dimension $d=1$, we adopt \eqref{reduced_CS_rho}--\eqref{reduced_CS_j} in $d=1$ instead, as this form allows for stronger analytical results. From a purely modelling perspective, it is noted that both systems \eqref{reduced_CS_rho_d}--\eqref{reduced_CS_j_d} and \eqref{reduced_CS_rho}--\eqref{reduced_CS_j} are consistent with the particle system once the system is in the flocking regime.
\end{remark}

\subsubsection{SPDE model for simple particle noise}
\label{subsubsec:SPDEderivation}
It is often the case that the deterministic dynamics of the Cucker-Smale model are enriched by the addition of noise terms (see \cite{10.1214/18-AAP1400} and the references therein). The inclusion of noise in the particle system results in the reduced model being a stochastic PDE. 
For example, let us consider adding a simple multiplicative noise  to the system with identical Brownian motion $B(t)$, as in \cite[System (1.8)]{10.1214/18-AAP1400}, for $d=1$
%\ana{what is here $\overline{v}$-it is pathwise def, so it is a RV defined as before, without the expectation}
\begin{align}
    dx_i & = v_i(t) dt \label{eq:stoc_det_x} , \\
    dv_i & = N^{-1}\sum_{j=1}^{N}{a(x_i(t)-x_j(t))(v_j(t)-v_i(t))} dt + \sigma (v_i(t) - \overline{v}) d B(t)\label{eq:stoc_det_v} , 
\end{align}
where $\sigma >0 $. Note that this extension still preserves the mean velocity of the system. The addition of the noise term similarly as before, gives a \emph{reduced stochastic Cucker-Small} PDE model 

\begin{align}
    d \rho(x,t) & = - \nabla \cdot j dt, \label{reduced_SCS_rho}\\
    d j(x,t) & = \left\{\rho (a\ast j) dt - j (a \ast \rho) - \overline{v}^2\nabla \rho\right\} dt + \sigma (j-\overline{v}\rho)dB(t) \label{eq:reduced_SCS_j}.
\end{align}

The inclusion of the noise in the PDE extends similarly to $d>1$. In Subsection \ref{subsec:SPDE} we present numerical results demonstrating that the reduced stochastic PDE reproduces the statistics of the stochastic Cucker-Smale model. 
The analysis of fluctuations via reduced SPDE models (including, but not limited to, \eqref{reduced_SCS_rho}--\eqref{eq:reduced_SCS_j}) is deferred to future works. 
%\sebastian{Mention that in other reduced models it is not as straightforward to incorporate noise?}

\subsection{Consistent choice of initial condition for reduced PDE models}
\label{subsec:choice_intial_cond}

Since, by construction, the masses of the empirical densities \eqref{rho_j} are one and $\overline{\f{v}}$ \eqref{particle_flocking_velocity}, respectively, for all times, it is natural to require the same for our PDE models. This is easily done by starting the PDE models from initial data $\rho_0(x)$ and $j_0(x)$ with masses one and $\overline{\f{v}}$ (respectively), and noticing that, since the interaction function $a$ is symmetric, such masses are preserved over time. In other words, we get
    \begin{align}\label{masses}
    \int_{\mathbb{T}^d}{\rho(\f{x},t)d\f{x}}\equiv 1,\qquad \int_{\mathbb{T}^d}{\f{j}(\f{x},t)d\f{x}} = \overline{\f{v}} :=\int_{\mathbb{T}^d}{\f{j}_0(\f{x})d\f{x}}, \qquad \forall t\geq 0.
    \end{align}
  %  \ana{maybe add somewhere a comment that for us $\overline{v}$ is a parameter since this I was asked few times}
\begin{remark}
In practice, it will sometimes be convenient to work with smoothed versions of the particles' empirical densities. Namely, we will consider the smoothed densities
\begin{align}\label{smoothed_densities}
        \rho_\epsilon(\f{x},t) := N^{-1}\sum_{i=1}^{N}{\delta_\epsilon(\f{x}-\f{x}_i(t))},\qquad j_\epsilon(\f{x},t) := N^{-1}\sum_{i=1}^{N}{\f{v}_i(t)\delta_\epsilon(\f{x}-\f{x}_i(t))} ,
    \end{align}
    where $\delta_\epsilon\colon \mathbb{T}^d\rightarrow \mathbb{R}$ is the von Mises kernel\footnote{The von Mises kernel $\delta_\epsilon\colon\mathbb{T}\rightarrow \mathbb{R}$ with variance $\epsilon^2$ is given by $
        \delta_\epsilon(x) := Z_\epsilon^{-1} \exp\{- \sin^2(x/2)/(\epsilon^2 /2)\}$, where $Z_\epsilon := \int_{\mathbb{T}} \exp\{- \sin^2(x/2)/(\epsilon^2 /2)\} dx$.} with variance $\epsilon^2$. %\sebastian{If we are instead working in $\mathbb{R}^d$ we would consider Gaussian kernels.}.
\end{remark}

\subsection{Notion of PDE flocking}

We specify what we mean by 'flocking' at the PDE level. In the following, we will simply write spaces without specifying the domain and  will denote the $L^2(\mathbb{T}^d)$-norm by $\| \cdot \|$. 

\begin{defn}[PDE flocking]\label{def:PDEflocking}
    For either one of the PDE models in Subsections \ref{subsub:1d}--\ref{subsub:1dweighted}--\ref{subsub:d_gr_1}, we say that PDE flocking occurs if
    \begin{align}\label{continuous_alignment}
    \lim_{t\rightarrow\infty}{\|\f{j}(t)-\overline{\f{v}}\rho(t)\|} = 0.    
    \end{align}
\end{defn}

Definition \ref{def:PDEflocking} essentially says: 
\emph{flocking at the continuous level is characterised by the momentum density $\f{j}$ aligning with a multiple of the density $\rho$, the multiple being the particle flocking velocity $\overline{\f{v}}$}. While we are interested in such an alignment measured in the $L^2$-norm, the same alignment measured in a weaker norm can be justified easily and intuitively at the level of the empirical densities \eqref{rho_j} when particle flocking occurs. 
%We can prove that \eqref{continuous_alignment} holds when particle flocking occurs in the following way. 
%\begin{lemma}\label{flockSPDE_intuition}
%Let $\rho,j$ be given as in \eqref{rho_j}. Assume particle flocking, 
For example, for $d=1$, assuming particle flocking in the form of $\lim_{t \rightarrow \infty}\mathbb{E}[|v_i(t)-\overline{v}|^2]=0$ implies $\lim_{t \rightarrow \infty}\mathbb{E}{|\langle \overline{v}\rho(t) - j(t), \varphi \rangle|} = 0$ for all $\varphi \in L^{\infty}(\mathbb{T})$. 
%\end{lemma}
%\begin{proof}
This is easily seen by using the definitions \eqref{rho_j} and noticing that as  $t\rightarrow \infty$, we have
    \begin{align*}
    & \mathbb{E}[|\langle \overline{v}\rho(t) - j(t), \varphi \rangle|]  \leq N^{-1}\sum_{i=1}^{N}{\mathbb{E}[|(v_i(t) - \overline{v})\varphi(x_i)|]} \leq N^{-1}\|\varphi\|_{L^\infty}\sum_{i=1}^{N}\mathbb{E}[|v_i(t)-\overline{v}|^2]^{\frac{1}{2}} \rightarrow 0.
    \end{align*}
 \begin{remark}
     Definition \ref{def:PDEflocking} can be seen an the inertial analogue of flocking conditions for continuous models derived by reducing the full mean-field-limit dynamics, see for instance \cite[Section 1.1.3]{Karper2013}.
 \end{remark}
%\end{proof}

To illustrate the flocking in both the particle and the PDE model, we perform numerical simulations of a two-dimensional particle model and compare it to the PDE result. The interaction function is given by \eqref{eq:NonConstComm}, with $\lambda =50$ and $r = 1/2$ and the initial positions and velocities of the particles are randomly generated from a uniform distribution on the interval $\f{x}_i \in [-20,20]^2$ and $\f{v}_i \in [-10,10]^2$, respectively. The initial data for the PDE is chosen to be the smoothed empirical density of the particles (see \eqref{smoothed_densities}), with $\epsilon = 1$. 
In the top row of Fig.~\ref{fig:J_vector}, we show snapshots from a single realisation of the two-dimensional particle model for $N=10^3$ particles at three different time points $t=0$, $t=0.8$ and $t=2$. In the middle and the bottom row of Fig.~\ref{fig:J_vector}, we plot the corresponding solutions of the reduced PDE. For the particle model we can see that while the velocities are initially randomly distributed, as the system flocks they start aligning with the mean velocity. For the PDE, the profile of $\rho$ changes initially, but at later times, once the system has flocked sufficiently, it is only translated in the direction of the mean velocity. Flocking in the PDE sense is evident when looking at the vector field of $\f{j}$, as the vector field starts to align in the direction of $\overline{\f{v}}$ (the red arrow with base at the origin), with the magnitude of the vectors at each position in space scaled by the density $\rho$ at that point.

\begin{figure}[h]
    \includegraphics[width=0.99\linewidth]{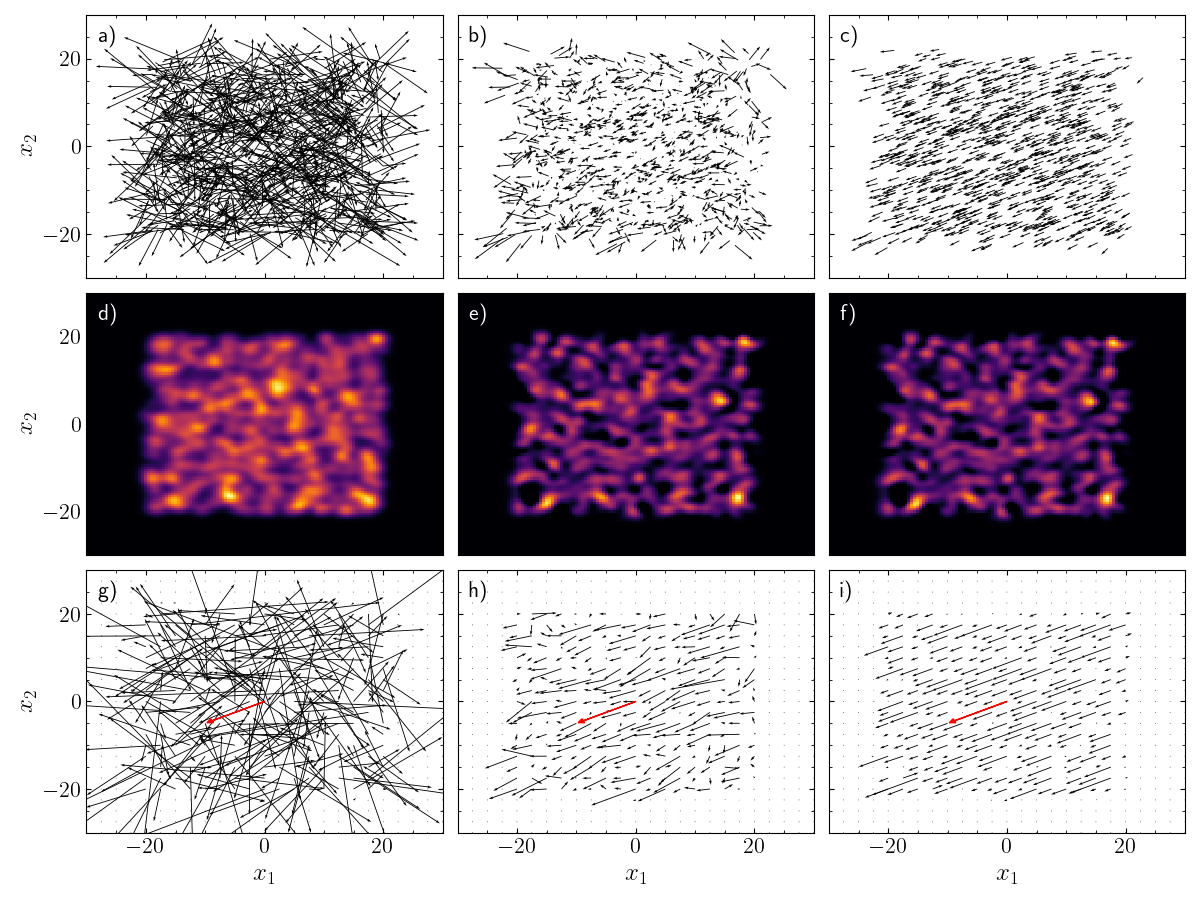}
\caption{Snapshots of the results obtained for the Cucker-Smale model \eqref{particles_det_x}-\eqref{particles_det_v} and reduced PDE model \eqref{reduced_CS_rho_d}-\eqref{reduced_CS_j_d} in two-dimensions at times $t=0$ (first column), $t=0.4$ (second column), and $t=2$ (third column). Top-row: particle system for $N=1000$, with the arrows and their base corresponding to the particles' velocities and positions respectively. Middle-row: the empirical density $\rho$ from the reduced model, brighter colours indicating a higher density. Bottom-row: the empirical momentum density represented as a vector field. The red arrow centred at the origin corresponds to $\overline{\f{v}}$.}  
\label{fig:J_vector}
\end{figure}

\section{Statements of the main results}
\label{sec:InformalMain}

Our first main result is composed of several results about PDE flocking, and is the following.

\begin{theorem}[PDE flocking: Informal statement combining Proposition \ref{local_result}, Proposition \ref{local_result_w}, Lemma \ref{well_posedness_reg}]
\label{thm:all_PDE_flocking_results}
Consider the Cucker-Smale dynamics given in Subsection \ref{subsec:CSmodel} with the class of interaction potentials given by \eqref{potential_split}. Then we have the following results concerning the associated PDE flocking (as per Definition \ref{def:PDEflocking}). 
\begin{description}
    \item[a.] (PDE flocking in $d=1$, cfr. Proposition \ref{local_result}). Subject to suitable smallness assumptions on the initial data, the 1d PDE model \eqref{reduced_CS_rho}--\eqref{reduced_CS_j} exhibits PDE flocking.
    \item[b.] (PDE flocking in $d=1$ with space-dependent weight $w$, cfr. Proposition \ref{local_result_w}). Subject to suitable smallness assumptions on the initial data and the weight $w$, the PDE model \eqref{reduced_CS_rho_w}--\eqref{reduced_CS_j_w} exhibits PDE flocking.
    \item[c.] (PDE flocking in $d>1$, cfr. Lemma \ref{well_posedness_reg}). Subject to suitable smallness assumptions on the initial data and to the diffusive-type regularisation \eqref{reg_reduced_CS_rho_d}--\eqref{reg_reduced_CS_j_d}, the PDE model \eqref{reduced_CS_rho_d}--\eqref{reduced_CS_j_d} exhibits PDE flocking.
\end{description}
\end{theorem}
%\begin{informal_result_closure}[Informal statement of Theorem \ref{thm_error_PDE_particles_H_minus_2}] \federico{produce informal statement}
%\end{informal_result_closure}
Our second main result quantifies the error  between particle system and PDE (in dimension $d=1$, for the unweighted version \eqref{reduced_CS_rho}--\eqref{reduced_CS_j}). To this end, we compare the solution $(\rho,j)$ of the PDE model \eqref{reduced_CS_rho}--\eqref{reduced_CS_j} to the smoothed versions of the empirical densities \eqref{smoothed_densities} evaluated at time $t$. We will denote the dual space of $H_0^2$ by $H^{-2}$. In order to evaluate the error between particle system and PDE, we will be using the space $\mathbf{H}^{-2} := H^{-2} \times  H^{-2}$, equipped with the norm $    \norm{(u,v)}_{\mathbf{H}^{-2}} := \overline{v}^2 \norm{u}_{H^{-2}}^2 + \norm{v}_{H^{-2}}^2.$ Our second main result reads as follows.

\begin{theorem}[Pathwise error between $\epsilon$-smoothed CS particle system and PDE system, $H^{-2}$ system]\label{thm_error_PDE_particles_H_minus_2}
    Let $d=1$ and $\epsilon>0$ be fixed. 
    Let $X_{PDE}$, $X_{CS}$ be defined as 
\begin{align*}
    X_{PDE} & = (\rho , j) \quad \text{with }\rho,j\text{ solving \eqref{reduced_CS_rho}--\eqref{reduced_CS_j}}\\
    X_{CS} & = (\rho_\epsilon , j_\epsilon) \quad \text{with }\rho_\epsilon, j_\epsilon\text{ as in \eqref{smoothed_densities}}.
\end{align*}
where $\rho_{\epsilon}$ satisfies $\|\rho_\epsilon(\cdot,0)\| \leq C_a/(4\theta \|g\|)$, and with the particles' initial velocities $\{v_i(0)\}_{i=1}^{N}$ all falling in $B_{R_v}(0)$ (see \eqref{particle_flocking_bound}).   
Assume the interaction kernel $a$ has bounded first derivative, and that the Fourier modes $\hat{a}$ satisfy 
\begin{align}
    | \hat{a}^{(r)} (\xi) | & \lesssim (1+ |\xi|^2)^{-1},  
    \label{bound_fourier_a}  \quad\,\, r=0,1,2, \,\,\forall \xi \in \mathbb{R}.
\end{align}
%Furthermore, assume that PDE and particles are started from the same initial data, i.e., that $\rho(\cdot,0) = \rho_\epsilon(\cdot,0)$ and $j(\cdot,0) = j_\epsilon(\cdot,0)$.

Then, we have the PDE/particles bound
\begin{align}\label{main_error_estimates}
   & \norm{X_{PDE}(t) - X_{CS}(t)}_{\mathbf{H}^{-2}}  \nonumber\\
   & \lesssim \norm{X_{PDE}(0) - X_{CS}(0)}_{\mathbf{H}^{-2}}C(a, \|j_0 - \overline{v}\rho_0\|) \nonumber\\
    & \quad + C(a, \|j_0 - \overline{v}\rho_0\|)\int_{0}^{t}{R_v e^{-a(x_M)s}(1+\sqrt{\epsilon}) + \left\{ \|j_0 - \overline{v} \rho_0\| e^{-C_as} + R_v e^{-a(x_M)s} \right\} ds},
\end{align}
where $R_v,x_M$ have been introduced in \eqref{particle_flocking_bound}.
\end{theorem}

\begin{remark}

We conducted the proof of Theorem \ref{thm_error_PDE_particles_H_minus_2} in the $\mathbf{H}^{-2}$ norm, as this crucially allows us to avoid using the propagation of chaos bounds for the underlying particle system (\cite[Theorem 1.2]{nguyen2022propagation}). These propagation of chaos bounds are diverging in time, and we do not see an obvious way to use them to get a uniform-in-time estimate of the likes of \eqref{main_error_estimates} if we work in norms stronger the $\mathbf{H}^{-2}$. We defer further investigations on the matter to future works.
\end{remark}

\begin{remark} 
While Theorem \ref{thm_error_PDE_particles_H_minus_2} is the clean theoretical statement for the mismatch between PDE and particle system given in the $\mathbf{H}^{-2}$-norm, our numerical simulations also consider stronger norms (up to $L^2$). The evaluation of stronger norms, in conjunction with the presence of the smoothed dirac deltas $\delta_\epsilon$ shows a (naturally expected) dependence of the type $N^{-1}\epsilon^{-\beta}$ for some $\beta>0$, in analogy to several considerations given in \cite{cornalba2021well}.
\end{remark}

\section{PDE flocking: Proof of Theorem \ref{thm:all_PDE_flocking_results}}
\label{sec:PDEflock}

This section is devoted to proving all components of Theorem \ref{thm:all_PDE_flocking_results}.
It is convenient to treat the kernel $a$  in \eqref{eq:NonConstComm} as the sum of a ``constant'' part and a ``variable'' part, namely as
\begin{align}\label{potential_split}
    a(\f{x}) = C_a+\theta g(\f{x}) > 0, 
\end{align}
where $C_a > 0$, $g\in L^2$ and $\theta$  is a small parameter and we assume without loss of generality that $\theta \geq 0$. Although not unique, the representation \eqref{potential_split} will make our proofs conceptually more intuitive.

We start with dimension $d=1$, where the analytical tools at our disposal are stronger. In subsection \ref{subsec:1Dflock}, we first treat the case of constant kernel $a=C_a$ and then move on to consider the general $a$ as in \eqref{potential_split}. In subsection \ref{subsec:Weights}, we provide flocking estimates (again in dimension $d=1$), but for the weighted system \eqref{reduced_CS_rho_w}--\eqref{reduced_CS_j_w}. Finally, we devote Subsection \ref{sec_flock_PDE_high_dim} to proving flocking in dimension $d>1$ for a physically relevant, diffusive-type regularised version of \eqref{reduced_CS_rho_d}--\eqref{reduced_CS_j_d}.

\subsection{PDE flocking in dimension one}
\label{subsec:1Dflock}

\subsubsection{The case of constant kermel $a=C_a$}

 %{For a constant, positive interaction, the particle system will most definitely flock.}
In this setting, several steps of the proof are simplified. Presenting the key ideas in this case is intended to aid the reader in grasping the more general form explained below. 
\begin{lemma}[PDE flocking with constant interaction function $a = C_a$]\label{thm_flocking_alignment}
Let $(\rho,j)$ be the solution to \eqref{reduced_CS_rho}--\eqref{reduced_CS_j} started from $(\rho_0,j_0)$.  Furthermore, assume that the interaction function $a$ is constant $a = C_a$.  Then we have the exponential decay
\begin{align}\label{flocking_alignment}
    \|j(t) - \overline{v} \rho(t)\|^2 \leq \exp\{-2C_at\}\|j_0 - \overline{v} \rho_0\|^2,
\end{align}
that implies PDE flocking.
\end{lemma}

%This means that $\rho$ and $j$ tend, in the limit $t\rightarrow +\infty$, to be multiple of each other by the constant $\overline{v}$, and this is and indication of flocking. {\color{red}[SZ: It seems like we do not use the same norms in Proposition \ref{flockSPDE_intuition} and this theorem.]}
\begin{proof}
The existence of a unique global-in-time solution, $L^2\times L^2$-valued solution to \eqref{reduced_CS_rho}--\eqref{reduced_CS_j} is granted by the linearity of the equation (entailed by the constant interaction function $a$) as well as from the well-posedness analysis similar to that presented in \cite[Section 4]{Cornalba2019}. 
    Writing the differential of the process $j(t) - \overline{v}\rho(t)$, we obtain
    \begin{align}
    \partial_t (j - \overline{v}\rho) 
    & = \rho (a\ast j)  
    - j (a \ast \rho) 
    - \overline{v}^2\nabla\rho
    + \overline{v}\nabla \cdot j.
    \end{align}
    Therefore,
    \begin{align}\label{diff_square}
     \frac{1}{2}\partial_t\| j - \overline{v}\rho \|^2 & =      \langle \rho (a\ast j)  
    - j (a \ast \rho) 
    - \overline{v}^2\nabla\rho
    + \overline{v}\nabla \cdot j,
    j - \overline{v}\rho\rangle \nonumber\\
    & \stackrel{\eqref{masses}}{=} 
      C_a\langle \overline{v}\rho 
    - j ,
    j - \overline{v}\rho\rangle 
    + 
    \langle 
    - \overline{v}^2\nabla\rho
    + \overline{v}\nabla \cdot j,
    j - \overline{v}\rho\rangle \nonumber\\
  %  &  = - C_a\| j - \overline{v}\rho \|^2
  %  + 
  %  \langle 
   % - \overline{v}^2\nabla\rho
   % + \overline{v}\nabla \cdot j,
   % j - \overline{v}\rho\rangle \nonumber\\
    &  = - C_a\| j - \overline{v}\rho \|^2
    + 
    \langle 
    - \overline{v}^2\nabla\rho,j\rangle +
        \langle 
    - \overline{v}^2\rho,\nabla \cdot j\rangle
    + \langle \overline{v}\nabla \cdot j,j \rangle
    + \langle \overline{v}^3 \nabla \rho,\rho \rangle.  %\\
   % & =: - C\| j - \overline{v}\rho \|^2 + \sum_{i=1}^{4}{A_i}.
    \end{align}
    Integration by parts implies that the sum of the last four terms is zero and an application of Gr\"onwall's estimate in differential form concludes the proof.
\end{proof}

\subsubsection{The case of general kernel $a=C_a +\theta g$}
The following proposition extends the validity of Lemma \ref{thm_flocking_alignment} by considering a suitable perturbation argument (entailed by the non-constant component $\theta g$ of the interaction function). 
\begin{prop}[PDE flocking with  interaction kernel $a = C_a +\theta g$]\label{local_result}
Let $(\rho,j)$ be the solution to \eqref{reduced_CS_rho}--\eqref{reduced_CS_j} with initial data $(\rho_0,j_0)$.  
Assume that the  interaction function has the form
 $a = C_a + \theta g$ >0, where $g\in L^{2}(\mathbb{T})$ and $\theta >0 $. %are such that $\theta\|g\| < C_a$ \sebastian{Can the last $<$ be $\leq$? Since we anyway have $a = C_a + \theta g$ >0.}%\ana{Why again do we need this - otherwise the interaction can touch zero}. 
 Furthermore, assume 
\begin{align}\label{ass_boundedness_rho} 
K := \left( 2\frac{\overline{v}^2\|\rho_0\|^2 + \|j_0\|^2}{\overline{v}^2} + \frac{2\cdot 3^2\|j_0 - \overline{v}\rho_0\|^2}{\overline{v}^2} \right)^{1/2} < \frac{C_a}{4\theta\|g\|}. 
\end{align}
Then we have global well-posedness of \eqref{reduced_CS_rho}--\eqref{reduced_CS_j} in $L^2 \times L^2$, and the flocking bound
\begin{align}\label{partial_flocking_alignment}
    \|\overline{v}\rho(t)-j(t)\|^2 \leq \exp\{-C_at\}\|j_0 -\overline{v} \rho_0\|^2, \qquad \mbox{for all }t\geq 0.
\end{align}
\end{prop}
\begin{proof}[Sketch of proof of Proposition \ref{local_result}]
The proof consists of the following three steps that are proven in detail in the Appendix \ref{Appendix:local_result}:
\begin{description}
\item  I) The system \eqref{reduced_CS_rho}--\eqref{reduced_CS_j} admits a continuous-in-time, local $L^2\times L^2$-valued solution;
\item  II) The flocking bound \eqref{partial_flocking_alignment} holds up to the time 
\begin{align}\label{stopping_time}
 \tau := 
\left\{
\begin{array}{rl}
\inf\left\{t>0: \|\rho(t)\| > \frac{C_a}{4\theta\|g\|}\right\}, & \mbox{if infimum taken over non-empty set} \\
\infty, & \mbox{otherwise }
\end{array}
\right.
\end{align}
and
\item  III) The energy estimate  $\|\rho(t)\| < K$ holds -- at least -- for all $t\leq \tau$, where $K$ is defined by 
\eqref{ass_boundedness_rho}. %, and for some constant $K$ satisfying $K < \frac{C_a}{4\theta\|g\|}$. \ana{I am confused with the same $K$ here}
\end{description}
Then it is straightforward to show that $\|\rho(t)\| \leq K$ for all $t>0$: if this were not the case, point I) would imply that $K < \|\rho(z)\| < \frac{C_a}{4\theta\|g\|}$ for some $z\leq \tau$. But this contradicts III). As $\|\rho(t)\| \leq K < \frac{C_a}{4\theta\|g\|}$ for all $t>0$, point II) implies that \eqref{partial_flocking_alignment} holds for all $t >0$, ending the proof.
\end{proof}
\begin{remark}
    The assumption \eqref{ass_boundedness_rho} is nothing but a smallness assumption on how much (in relative size) the interaction function $a$ is allowed to deviate from a constant value. This is visible in the ratio $C_a/\theta \|g\|_{\infty}$. 
    Moreover, the bound \eqref{ass_boundedness_rho} is, in most realistic scenarios, essentially independent of $\overline{v}$: this can be seen by taking the relevant case $\overline{v}^2\|\rho_0\|^2 \propto \|j_0\|^2$.
    %\ana{do we want to compare something more about this wrt our simulations?}
\end{remark}

\subsubsection{Weight-dependent kinetic term}
\label{subsec:Weights}

In this subsection we justify and analyze the modified model \eqref{reduced_CS_rho_w}--\eqref{reduced_CS_j_w}.

Upon considering the $\epsilon$-mollification introduced in \eqref{smoothed_densities}, the derivation of the 1d reduced PDE model \eqref{reduced_CS_rho}--\eqref{reduced_CS_j} relies on the approximation 
\begin{align}\label{Approx_j2}
N^{-1}\sum_{i=1}^{N}{v_i(t)^2\delta'_\epsilon(x-x_i(t))} \approx N^{-1}\sum_{i=1}^{N}{\overline{v}^2\delta'_\epsilon(x-x_i(t))} = \overline{v}^2 \nabla \rho_\epsilon(x,t)
\end{align}
see Subsection \ref{subsub:1d}.
While the approximation \eqref{Approx_j2} is correct in the flocking regime, it can potentially be quite coarse for small times. 

An important observation is that the term $\textstyle N^{-1}\sum_{i=1}^{N}, {v_i^2\delta'_\epsilon(x-x_i)}$, despite its complicated appearance, is the gradient of the strictly positive function $\textstyle N^{-1}\sum_{i=1}^{N}{v_i^2\delta_\epsilon(x-x_i)}$. Since $\rho_\epsilon$ is trivially positive by definition \eqref{smoothed_densities}, we can rewrite the left-hand-side of \eqref{Approx_j2} as
\begin{align*}
N^{-1}\sum_{i=1}^{N}{v_i(t)^2\delta'_\epsilon(x-x_i(t))} = \overline{v}^2\nabla\left( w(x,t)\rho_\epsilon(x,t)\right) ,
\end{align*}
where we have defined
\begin{align}\label{DefnWeight}
w(x,t) := \frac{N^{-1}\sum_{i=1}^{N}{v_i^2\delta_\epsilon(x-x_i(t))}}{\overline{v}^2\rho_\epsilon(x,t)} > 0.
\end{align}
Essentially, the weight $w$ provides an algebraic compensation to the model in the pre-flocking regime. In practice, computing the exact $w$ \eqref{DefnWeight} is inconvenient (it corresponds to simulating the particles, which is what we want to avoid), and this may only be feasible for the initial condition $w_0(x) = w(x,0)$ since, at time $t=0$, the particles' positions and velocities are either given or need to be simulated regardless.
Nonetheless, it makes sense to consider $w$ as a variable capable of mimicking relevant (less computationally demanding) properties of the particle system \emph{prior} to the flocking regime. 
Said differently, the hope is that, for some `properly tuned' weight $w$, the modified PDE \eqref{reduced_CS_rho_w}--\eqref{reduced_CS_j_w}
will be an even better approximation of the particle dynamics than \eqref{reduced_CS_rho}--\eqref{reduced_CS_j} as a result of the improved pre-flocking modelling given by the weight $w$ (so far, we have used the flocking approximation for \emph{all} times, i.e., $w\equiv	1$).

In this section, we show that we can replicate  the global well-posedness analysis and flocking property for \eqref{reduced_CS_rho_w}--\eqref{reduced_CS_j_w} (Proposition \ref{local_result_w}) for a relevant class of weights (uni-vocally determined once the particles' initial configuration is known), see Definition 
\ref{defn_admissible_weights} below. Furthermore, we numerically show that using a simple exponentially-converging weight $w$ significantly improves the PDE accuracy.

\begin{defn}[Admissible weights $w$]\label{defn_admissible_weights}
A weight function $w\colon \mathbb{T} \times [0,\infty) \rightarrow [0,\infty)$ is called an \emph{admissible weight} if the following hold:
\begin{itemize}
\item Regularity: $w \in C^{1,1}$ and there exist two constants $w_{max} \geq w_{min} > 0$ such that $w_{min} \leq w \leq w_{max}$ for all $t\geq 0$ and $x\in\mathbb{T}^d$.
\item Long-time behaviour: the weight $w$ agrees with one eventually, i.e., there exists $T_a<\infty$ such $w\equiv 1$ for all $t\geq T_a$.
\end{itemize} 
\end{defn}
Subject to $w$ being sufficiently small we prove the following global well-posedness and flocking result for \eqref{reduced_CS_rho_w}--\eqref{reduced_CS_j_w}, whose proof is presented in Appendix \ref{Appendix:weights_flocking}.

\begin{prop}[PDE flocking for weight-dependent setting]\label{local_result_w}
Consider $a = C_a + \theta g >0$ as in \eqref{potential_split}, where $g\in L^{2}$  and $\theta >0 $. Let $w$ be an admissible weight. Furthermore, assume that:
\begin{description}
\item [C1] The norms $\|1-w\|_{\infty}, \|\partial_t w\|_{\infty}, \|\nabla w\|_{\infty} $ are sufficiently small so that the following bound 
\begin{align}\label{ass_boundedness_rho_w} 
\overline{v}^2Q^2 > K & := 2\frac{w_{\max}}{w_{\min}}\eta(0)\exp\left\{\int_{0}^{\infty}{ \|\partial_t w(s)\|\frac{1}{2w_{\min}} ds }\right\} \nonumber\\
 & \quad + \frac{ 2\cdot 6^2 }{w_{\min}^2} \left( \sqrt{|R_0|} + \int_{0}^{t}{\sqrt{\beta(s)}ds} \right)^2 \exp\left\{ \int_{0}^{\infty}{ \|\partial_t w(r)\|_{\infty}\frac{1}{2w_{\min}}dr} \right\}, 
\end{align}
%\ana{I am a bit bothered with this huge expression for $K$ I think we can write it precisely in the appendix and maybe here just write:
%the bound
%\[
%K=K(\frac{w_{\max}}{w_{min}},\eta(0),\overline{v}, \partial_t w, \sqrt{\beta(s)},\langle j_0 - \overline{v}\rho_0, j_0 - \overline{v}w_0\rho_0 \rangle ) <  \overline{v}^2Q^2
%\]
%where $K$ is precisely defined by (ref in Appendix)}\natasa{I agree.}
holds, for some fixed $Q < C_a/{8\theta\|g\|_\infty}$, where we have set
\begin{align*}
\eta(t) & := \overline{v}^2\|\rho(t)\|^2 + \|j(t)\|^2 \\
\beta(t) & := 2e^{-C_at}\int_{0}^{t}{e^{C_as} \tilde{\beta}(s) ds} + \overline{v}^2\|w-1\|^2Q^2, \\
\tilde{\beta} & := 2\overline{v}^2 Q^2 \left(\overline{v}\|\nabla w\|_{\infty} + \|\partial_t w\|_{\infty} \right) + \overline{v}^2\|\theta g\|Q^3\|1-w\|_{\infty}.
\end{align*}
\item [C2]
The property $\beta(t)\rightarrow 0$ as $t\rightarrow \infty$ holds.
\end{description}
Then we have global well-posedness of \eqref{reduced_CS_rho_w}--\eqref{reduced_CS_j_w} in $L^2 \times L^2$, and the flocking bound  
\begin{align}\label{partial_flocking_alignment_b}
\|j(t)-\overline{v}\rho(t)\|^2 & \leq 2|\langle j_0 - \overline{v}\rho_0, j_0 - \overline{v}w_0\rho_0 \rangle| e^{-C_a t} + \beta(t),
\end{align}
which holds for all $t>0$.
\end{prop}
\begin{remark}\label{smallness_of_w} Despite the apparent convoluted appearance of some constants in Proposition \ref{local_result_w}, note that the conditions {\bfseries C1--C2}
can always be satisfied by choosing weight $w$ with sufficiently small norms for $\|w-1\|_{\infty}, \|\partial_t w\|_{\infty}, \|\nabla w\|_{\infty} $. In fact, Proposition \ref{local_result_w} is a `continuous perturbation' of the corresponding result for $w\equiv 1$ (Proposition \ref{local_result}).
\end{remark}
While we defer quantitative results for the error between the weighted PDE \eqref{reduced_CS_rho_w}--\eqref{reduced_CS_j_w} to future work, we were able to numerically verify that a very simple choice of weight $w$ grants a significant experimental reduction of such error. Namely, the choice of $w$ that we adopt in the experiment is 
\begin{align}\label{exp_weight}
    w(x,t) = 1-e^{-C_a t} + e^{-C_a t}w_0(x) ,
\end{align}
where $w_0(x)$ is the exact weight at time zero. While this weight is not admissible, the intuition for \eqref{exp_weight} being a good weight is the following:
    \eqref{exp_weight} is exact at time zero, consistent in the flocking regime, and following exponential decay with a rate compatible with that of the PDE/CS flocking estimates.

As a qualitative demonstration of the impact this choice of weight has on the discrepancy between the particle system and reduced PDE, in Fig.~\ref{fig:3} we show results for the Cucker-Smale model as well as PDE with $w=1$ and weight \eqref{exp_weight}. In particular, we show $\rho(x,t=2)$ and $j(x,t=2)$ for the one-dimensional setting for initial data with a large velocity spread (uniformly distributed on the interval $[-50,50]$), with $\epsilon = 5$. We set $r = 1/2$ and $\lambda =50$ in the interaction potential \eqref{eq:NonConstComm}. As the velocity spread is large, we would not expect the reduced model to accurately reproduce the results of the particle system, which can be seen by the mismatch between the red and black lines in the figure. The inclusion of the weight \eqref{exp_weight} with $C_a = \lambda/2$ (blue line) can be seen to cause a significant increase in accuracy with no additional computational cost when computing the PDE. While this is only a qualitative example, in Fig.~\ref{fig:PDE_closing_comp} we present more extensive numerical results suggesting that the inclusion of the weight decreases the discrepancy between particle and reduced system by a factor of approximately five.

\begin{figure}[h]
    \includegraphics[width=0.99\linewidth]{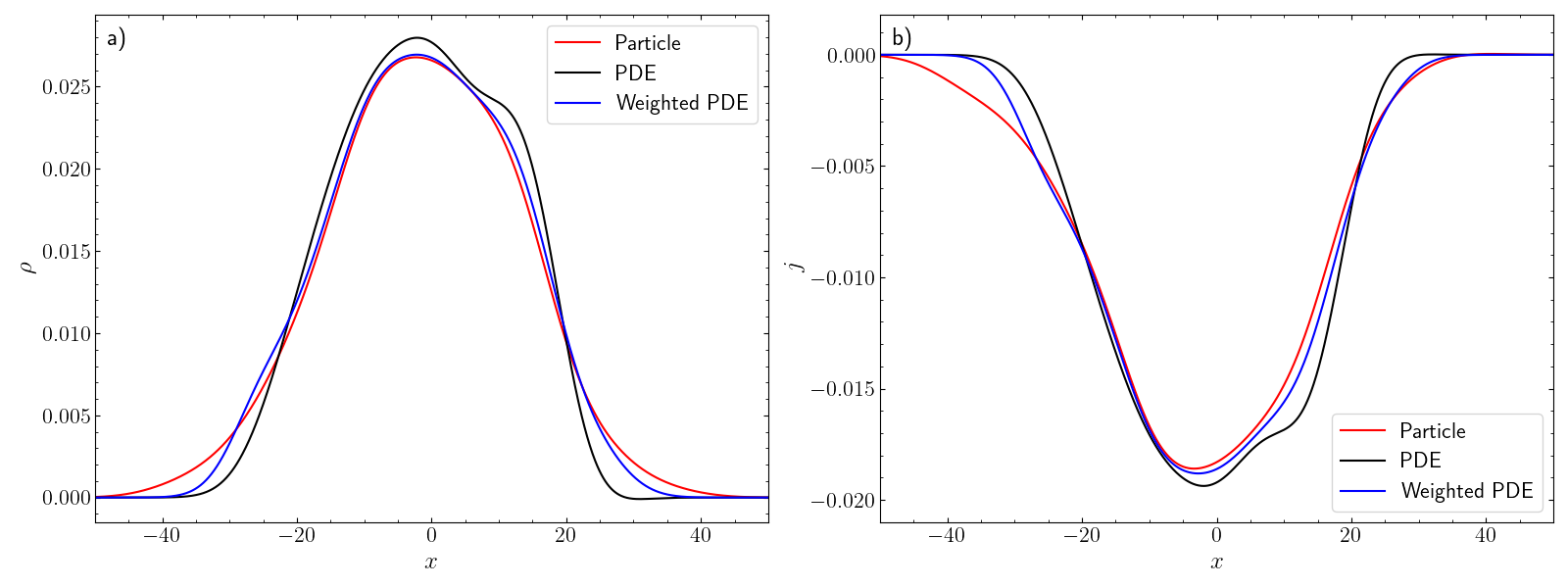}
\caption{Results obtained for the Cucker-Smale model \eqref{particles_det_x}-\eqref{particles_det_v}, reduced PDE \eqref{reduced_CS_rho}-\eqref{reduced_CS_j} and reduced PDE with the weight-dependent kinetic term \eqref{reduced_CS_rho_w}-\eqref{reduced_CS_j_w} in one dimension, where $w$ is given by \eqref{exp_weight}. a) The empirical density $\rho$ and  b) the empirical momentum density $j$, at $t=2$, evaluated for the particle system using a kernel with $\epsilon = 5$.} \label{fig:3}
\end{figure}

\subsection{PDE flocking in dimension $d>1$}
\label{sec_flock_PDE_high_dim}

In contrast to the one-dimensional case, and mainly due to the degeneracy of the divergence operator, it is not straightforward to prove both flocking and well-posedness for the PDE \eqref{reduced_CS_rho_d}--\eqref{reduced_CS_j_d}. % This is mainly due to integration by parts failing to cancel out specific terms in divergence form (namely, in contrast with what happens in dimension $d=1$, terms of the form $\int{\nabla \cdot \f{j}, j_{z}}$, $z=1,\dots,d$, do not vanish). 
For that reason, we introduce a PDE regularisation, which has the goal of keeping the norms $\|\rho\|$, $ \|\f{j}\|$ from getting too large (thus providing global well-posedness) while -- at the same time -- preserving PDE flocking.

For this regularisation, we introduce the following notation. Let $G_{h,d}$ be a grid of mesh size $h$ over the domain $\mathbb{T}^d$, and let $\mathcal{C}_{h,d}$ be the set of grid cells of $G_{h,d}$. 
For every $c\in\mathcal{C}_{h,d}$, we introduce the notation
\begin{align*}
\|f\|_{c} & := \|f\|_{L^2(c)}, \qquad
\|\rho\|_{\nabla,c}  := \|\nabla\rho\|_{c}, \qquad \|\f{j}\|_{\nabla,c} := \left(\sum_{z=1}^{d}{\|\nabla j_z\|_{c}^2}\right)^{1/2}, \\
\|(\rho,\f{j})\|_{\nabla,c} & := \left(|\overline{\f{v}}|^2 \|\rho\|_{\nabla,c}^2 + \|\f{j}\|^2_{\nabla,c} \right)^{1/2}.
\end{align*}
Furthermore, we introduce a smooth cut-off function $\varphi:[0,+\infty)\rightarrow [0,+\infty) $ 
\begin{align}\label{cutoff}
\varphi(y) = 
\left\{
\begin{array}{ll}
0, & \mbox{if } y < V h^d, \\
\mbox{smooth transition}, & \mbox{if }y \in [V h^d, (V+1)h^d), \\
W, & \mbox{if } y \geq (V+1)h^d,
\end{array}
\right. 
\end{align}
for some constants $V, W$ to be chosen. Finally, let $e_c$ be a hat function with unitary mass, and with support on a neighbourhood of cell $c$. In particular, this implies that $e_c(x)\gtrsim h^{-d}$ on cell $c$.

The regularised PDE then reads
\begin{align}
    \partial_t \rho(x,t) & = - \nabla \cdot \f{j}  
    + \sum_{c\in\mathcal{C}_{h,d}}{\varphi(\|\rho,\f{j}\|_{\nabla,c})\nabla\cdot\left(e_c\nabla\rho\right)}\label{reg_reduced_CS_rho_d},\\
    \partial_t j_m(x,t) & = \left\{\rho (a\ast j_m)  - j_m (a \ast \rho) - \overline{v}_m \nabla \cdot \f{j}\right\}  \nonumber\\
	& \quad \quad + \sum_{c\in\mathcal{C}_{h,d}}{\varphi(\|\rho,\f{j}\|_{\nabla,c})\nabla\cdot\left(e_c\nabla j_m\right)}, \qquad m=1,\dots,d.   
    \label{reg_reduced_CS_j_d}
\end{align}
\begin{remark}
In practical terms, the grid $G_{h,d}$ will often coincide with the grid which we use in our numerical discretisations, so this formulation is convenient.
\end{remark}
%\begin{remark} \ana{if we need space, remove this remark}
%The effect of the regularisation in the PDE \eqref{reg_reduced_CS_rho_d}--\eqref{reg_reduced_CS_j_d} is to introduce diffusion for $\rho$ and $\f{j}$ locally, on all cells $c\in\mathcal{C}_{h,d}$ on which $\nabla\rho$ or $\nabla\f{j}$  gets too large.
%While, ideally, we would only want to use the density $\rho$ as the argument for regularisation, analytical technicalities currently require us to consider both $\rho$ and $\f{j}$ (these are, in any case, aligned during flocking). Removing $\f{j}$ as an argument of the regularisation in \eqref{reg_reduced_CS_rho_d}--\eqref{reg_reduced_CS_j_d} is deferred to future works. %\ana{why would we want this?}. 
%\end{remark}

First, we show that the PDE \eqref{reg_reduced_CS_rho_d}--\eqref{reg_reduced_CS_j_d} satisfies the flocking property.
\begin{lemma}\label{flocking_property_multi_d}
For $\rho,\f{j}$ solving \eqref{reg_reduced_CS_rho_d}--\eqref{reg_reduced_CS_j_d}, we have the bound $\|\f{j}(\cdot,t)-\overline{\f{v}} \rho(\cdot,t)\|^2 \leq e^{-C_a t}\|\f{j}_0(\cdot)-\overline{\f{v}} \rho_0(\cdot)\|^2$ up to the stopping time $\tau := \inf\{t>0: \|\rho(\cdot,t)\| \geq C_a/(4\theta\|g\|_\infty)\}$, where, as before, we have decomposed the interaction function $a = C_a +\theta g$.
\end{lemma}
\begin{proof}
The proof is in strict analogy with that of Proposition \ref{local_result}, part II). Namely, by taking the time differential of $j_m - \overline{v}_m\rho$ and testing with $j_m - \overline{v}_m\rho$, we obtain
\begin{align}\label{reg_flocking_neg_contribution}
 \partial_t \frac{1}{2}\|j_m - \overline{v}_m\rho\|^2 % \nonumber\\
%& \quad = 
% \left\langle \rho (a\ast j_m)  - j_m (a \ast \rho) - \overline{v}_m \nabla \cdot \f{j} + \overline{v}_m \nabla \cdot \f{j} , j_m - \overline{v}_m\rho \right\rangle \nonumber \\
% & \quad \quad + \left\langle \sum_{c\in\mathcal{C}_{h,d}}{\varphi(\|\rho,\f{j}\|_{\nabla,c})\nabla\cdot\left(e_c\nabla j_m\right)} - \sum_{c\in\mathcal{C}_{h,d}}{\varphi(\|\rho,\f{j}\|_{\nabla,c})\nabla\cdot\left(e_c\nabla \overline{v}_m\rho\right)}, j_m - \overline{v}_m\rho \right\rangle \nonumber \\
 &= 
 \left\langle \rho (a\ast j_m)  - j_m (a \ast \rho) , j_m - \overline{v}_m\rho \right\rangle \nonumber \\
 & \quad  - \left\langle \sum_{c\in\mathcal{C}_{h,d}}{\varphi(\|\rho,\f{j}\|_{\nabla,c}) e_c\nabla (j_m - \overline{v}_m\rho)} , \nabla(j_m - \overline{v}_m\rho) \right\rangle \\
 &  \leq \left\langle \rho (a\ast j_m)  - j_m (a \ast \rho) , j_m - \overline{v}_m\rho \right\rangle,  \nonumber
\end{align}
where we used that the term \eqref{reg_flocking_neg_contribution} is non-positive. The proof is concluded by reusing the same arguments as in Proposition \ref{local_result}, point II), and summing up over $m=1,\dots,d$.
\end{proof}
In particular, we see that the regularisation introduced does not affect the flocking properties of our PDE model. We now need to show that the regularisation also gives us global well-posedness. In particular, we need (as in other previous arguments) a uniform bound of the type
$
\|\rho\| < C_a/(4\theta\|g\|_{\infty}),
$ 
%(where, as usual, we have decomposed the interaction function $a = C +\theta f$)
which will then allow us to extend the flocking property of Lemma \ref{flocking_property_multi_d} to all times $t>0$. This is in total analogy to the reasoning of Proposition \ref{local_result}, point III). 

\begin{lemma}\label{well_posedness_reg}
Assume the existence of $\gamma>0$ such that 
\begin{align}\label{properties_constant_A_1}
\frac{|\overline{v}|^2(\gamma^2 - 2C_P^2)}{2C^2_P} & > 1, \\
3\gamma^2 + \frac{2\|\f{j}_0 - \overline{\f{v}}\rho_0\|^2}{|\overline{\f{v}}|^2}  & < \left(\frac{C_a}{4\theta\|g\|_{\infty}}\right)^2,\label{properties_constant_A_3}
\end{align}
where $C_P$ is the Poincar\'e constant for the domain $\mathbb{T}^d$.
Then we can choose the constants $V,W$ in \eqref{cutoff} such that we have $\|\rho\| < C_a/(4\theta\|g\|_{\infty})$. 
%This implies, by Lemma \ref{flocking_property_multi_d}, 
and we have PDE flocking for \eqref{reg_reduced_CS_rho_d}--\eqref{reg_reduced_CS_j_d} over all times $t>0$.
\end{lemma}
We defer the proof of Lemma \ref{well_posedness_reg} to Appendix \ref{app:multi_d_flocking}.
\begin{remark}
Even though the constant in the left-hand-side of \eqref{properties_constant_A_3} has a multiplicative factor $|\overline{\f{v}}|^{-2}$, this is not a problem, since the component $\|\f{j}_0 - \overline{\f{v}}\rho_0\|^2$ is supposed, in realistic applications, to scale with $|\overline{\f{v}}|^2$, leading to a quantity of order one. This is in analogy to what we required in \eqref{ass_boundedness_rho}. Moreover, the threshold $V$ in \eqref{cutoff} scales with $|\overline{\f{v}}|^2$, as we can see from \eqref{choice_K}. This is also not a problem, as the same prefactor appears in front of $\|\nabla\rho\|_c^2$ in the definition of the activation quantity $\|\rho,\f{j}\|_{\nabla,c}$.
 Hence, in practice, subject to reasonable initial data $(\rho_0,\f{j_0})$ and big enough $C_a/(4\theta\|g\|_{\infty})$ (i.e. for a potential $a$ with not-too-large relative oscillations) the setting of Lemma \ref{well_posedness_reg} is not affected the scaling in $|\overline{\f{v}}|^2$.
\end{remark}

\label{subsec:nDflock}

\section{Error between reduced PDE and particle system: Proof of Theorem \ref{thm_error_PDE_particles_H_minus_2}}
\label{sec:ReductionAccuracy}

The proof of Theorem \ref{thm_error_PDE_particles_H_minus_2} also uses the following result, which quantifies the discrepancy caused by mollifying the Dirac deltas with smooth kernels $\delta_\epsilon$ in the definition of the particle empirical densities.

\begin{prop}[$\epsilon$-mollifying error in dimension $d=1$]\label{prop:Mollifying_err} Assume the kernel $a$ has bounded first derivative. Upon replacing the Dirac delta with von Mises kernels $\delta_\epsilon$, the term $T_1$ in \eqref{Evolution_J} is approximated by
\begin{equation}\label{smoothed_convolution_error}
T_1 = \rho_\epsilon (a\ast j_\epsilon) - j_\epsilon (a\ast \rho_\epsilon) + R_{\epsilon},
\end{equation}
where
\begin{align}\label{lem_smoothed_convolution_error}
R_\epsilon & \lesssim \left|\rho_\epsilon N^{-1}\sum_{j=1}^{N}{(v_j - \overline{v})r_{\epsilon,j}}\right| +  N^{-1}\sum_{i=1}^{N}{\delta_\epsilon(x-x_i) |x-x_i|\left(N^{-1}\sum_{j=1}^{N}{|v_j-\overline{v}}|\right)} \nonumber\\
& \quad +  \left|(j_\epsilon - \overline{v}\rho_\epsilon)N^{-1}\sum_{j=1}^{N}{r_{\epsilon,j}}\right|
%& \quad + N^{-1}\sum_{i=1}^{N}{\delta_\epsilon(x-x_i) |x-x_i|\left(N^{-1}\sum_{j=1}^{N}{|v_j-\overline{v}}|\right)} \nonumber\\
+ N^{-1}\sum_{i=1}^{N}{|v_i-\overline{v}|\delta_\epsilon(x-x_i)|x-x_i|}
\end{align}
where $\|r_{\epsilon,j}\|_{\infty}\lesssim \sqrt{\epsilon}$ for each $j=1,\dots, N$. Furthermore, we have
\begin{align}\label{bound_H_2_Residual}
\| R_\epsilon \|_{H^{-2}} \lesssim C(a)R_v e^{-a(x_M)t}\sqrt{\epsilon},
\end{align}
where $x_M$ is introduced in \eqref{particle_flocking_bound}.
\end{prop}
\begin{proof}
Upon replacing the Dirac deltas with the von Mises kernels $\delta_\epsilon$, adding and subtracting suitable quantities involving the limiting velocity $\overline{v}$, and performing the splitting $a(x_i(t)-x_j(t)) = a(x -x_j(t)) + [a(x_i(t)-x_j(t)) - a(x -x_j(t))]$, we get that $T_1$ in \eqref{derive_convolutions_atomic} is written as 
\begin{align*}
%& N^{-1}\sum_{i=1}^{N}{\delta_\epsilon(x-x_i(t)) N^{-1}\sum_{j=1}^{N}{a(x_i(t)-x_j(t))v_{j}(t)}}\nonumber \\
%& \quad\quad - N^{-1}\sum_{i=1}^{N}{v_{i}(t) \delta_\epsilon(x-x_i(t)) N^{-1}\sum_{j=1}^{N}{a(x_i(t)-x_j(t))}} \\
%& = N^{-1}\sum_{i=1}^{N}{\delta_\epsilon(x-x_i(t)) N^{-1}\sum_{j=1}^{N}{a(x_i(t)-x_j(t))(v_{j}(t) - \overline{v})}}\nonumber \\
%& \quad\quad - N^{-1}\sum_{i=1}^{N}{(v_{i}(t) - \overline{v})\delta_\epsilon(x-x_i(t)) N^{-1}\sum_{j=1}^{N}{a(x_i(t)-x_j(t))}} \\
T_1 & = N^{-1}\sum_{i=1}^{N}{\delta_\epsilon(x-x_i(t)) N^{-1}\sum_{j=1}^{N}{a(x-x_j(t))(v_{j}(t) - \overline{v})}}\nonumber \\
& \quad\quad + N^{-1}\sum_{i=1}^{N}{\delta_\epsilon(x-x_i(t)) N^{-1}\sum_{j=1}^{N}{[a(x_i(t)-x_j(t)) - a(x-x_j(t))](v_{j}(t) - \overline{v})}}\nonumber \\
& \quad\quad - N^{-1}\sum_{i=1}^{N}{(v_{i}(t) - \overline{v})\delta_\epsilon(x-x_i(t)) N^{-1}\sum_{j=1}^{N}{a(x-x_j(t))}} \\
& \quad\quad - N^{-1}\sum_{i=1}^{N}{(v_{i}(t) - \overline{v})\delta_\epsilon(x-x_i(t)) N^{-1}\sum_{j=1}^{N}{[a(x_i(t)-x_j(t)) - a(x-x_j(t))]}} =: \sum_{i=1}^{4}{Z_i}.
\end{align*}
Using the convolution approximation as in \cite[Lemma 3.6]{cornalba2020weakly}, 
we deduce that, for some functions $r_{\epsilon,j}$ such that $\|r_{\epsilon,j}\|_\infty \lesssim \sqrt{\epsilon}$, we obtain
\begin{align*}
Z_1 + Z_3 & = N^{-1}\sum_{i=1}^{N}{\delta_\epsilon(x-x_i(t)) N^{-1}\sum_{j=1}^{N}{(v_{j}(t) - \overline{v})\left[\int_{\mathbb{T}}{a(x-y)\delta_\epsilon(y-x_j(t))d y} + r_{\epsilon,j}\right]}} \\
& \quad - N^{-1}\sum_{i=1}^{N}{(v_{i}(t) - \overline{v})\delta_\epsilon(x-x_i(t)) N^{-1}\sum_{j=1}^{N}{\left[\int_{\mathbb{T}}{a(x-y)\delta_\epsilon(y-x_j(t))d y} + r_{\epsilon,j}\right]}}\\
& = \rho_\epsilon (a\ast [j_\epsilon - \overline{v}\rho_\epsilon]) - (j_\epsilon -\overline{v}\rho_\epsilon)a\ast \rho_\epsilon \\
& \quad \quad + \rho_\epsilon N^{-1}\sum_{j=1}^{N}{(v_{j}(t) - \overline{v})r_{\epsilon,j}} - (j_\epsilon - \overline{v}\rho_\epsilon)N^{-1}\sum_{j=1}^{N}{r_{\epsilon,j}}\\
& = \rho_\epsilon (a\ast j_\epsilon ) - j_\epsilon a\ast \rho_\epsilon + \rho_\epsilon N^{-1}\sum_{j=1}^{N}{(v_{j}(t) - \overline{v})r_{\epsilon,j}} + (j_\epsilon - \overline{v}\rho_\epsilon)N^{-1}\sum_{j=1}^{N}{r_{\epsilon,j}}.
\end{align*}
Furthermore, using a simple Taylor expansion for $a$, we get 
\begin{align*}
|Z_2 + Z_4| & \lesssim \|a'\|_\infty\left\{ N^{-1}\sum_{i=1}^{N}{\delta_\epsilon(x-x_i) |x-x_i|\left[N^{-1}\sum_{j=1}^{N}{|v_j-\overline{v}}|\right]} \right.\\
& \quad \quad \left. + N^{-1}\sum_{i=1}^{N}{|v_i-\overline{v}|\delta_\epsilon(x-x_i)|x-x_i|}\right\},
\end{align*}
and \eqref{lem_smoothed_convolution_error} is proven.
Finally, \eqref{bound_H_2_Residual} follows from using the characterisation of the $H^{-2}$-norm $\|R_{\epsilon}\|_{H^{-2}} = \sup_{\varphi \in H^2(\mathbb{T})\colon \|\varphi\|_{H^2(\mathbb{T})}=1}{|\int_{\mathbb{T}}{R_\epsilon \varphi|}}$, the one-dimensional embedding $H^{1}\subset L^\infty$, inequality \eqref{particle_flocking_bound}, the inequality $\sum_{i=1}^{N}{|b_i|} \leq N^{1/2}(\sum_{i=1}^{N}{|b_i|^2})^{1/2}$, the bound $\|r_{\epsilon,j}\|_{\infty}\lesssim \sqrt{\epsilon}$ and Gaussian moment bounds which promptly give us the estimate $\int_{\mathbb{T}}{\delta_{\epsilon}(x-x_i(t))|x-x_i(t)|d x} \lesssim \epsilon$. 
\end{proof}

\begin{proof}[Proof of Theorem 
\ref{thm_error_PDE_particles_H_minus_2}]

\emph{Step 1: functional set-up}.
In $d=1$, we have 
$    
\partial_t X_{PDE} = A X_{PDE} + F(X_{PDE}),
$
where 
\begin{align*}
    AX_{PDE} = A \begin{pmatrix}
        \rho \\
        j
    \end{pmatrix} 
    = \begin{pmatrix}
        - \nabla \cdot j \\
        - \overline{v}^2 \nabla \rho
    \end{pmatrix} , \qquad 
    F(X_{PDE}) = \begin{pmatrix}
        0 \\
        \rho (a*j) - j (a * \rho)
    \end{pmatrix} .
\end{align*}

Furthermore, $A$ generates a $C_0$-semigroup $(S_t)_{t \geq 0}$ of contractions in $\mathbf{H}^{-2}$ (see \cite[Lemma 4.2]{Cornalba2019}, and we can write $X_{PDE}$, $X_{CS}$ as 
\begin{align}
    X_{PDE}(t) &= X_{PDE}(0) + \int_0^t S_{t-s} F(X_{PDE}(s)) ds \label{PDE}\\
    X_{CS}(t) &= X_{CS}(0) + \int_0^t S_{t-s} F(X_{CS}(s)) ds + \int_0^t S_{t-s} R(s) ds \label{CS},
\end{align}
where $R(s) := R_{closing}(s) + R_{\epsilon}(s)$, where $R_{\epsilon}(s)$ is bounded as in Proposition \ref{prop:Mollifying_err}, and
\begin{align*}
    R_{closing}(s) & := \left( 0 \,; -N^{-1} \sum_{i=1}^N (v_i^2(s) - \overline{v}^2) \delta_\epsilon' (x-x_i(s))  \right).
\end{align*}

\emph{Step 2: bounding difference between PDE and CS with path-wise Gr\"onwall estimate}.
Subtracting \eqref{PDE} and \eqref{CS}, taking the $\mathbf{H}^{-2}$ norm and using the contractivity of $S(t)$, we get 
%\begin{align*}
%    X_1(t) - X_2(t) = X_1(0) - X_2(0) + \int_0^t \rho(t-s) \left[ F(X_1(s)) - F(X_2(s)) \right] ds + \int_0^t s(t-s) R(s) ds .
%\end{align*}

\begin{align}
    \norm{X_{PDE}(t) - X_{CS}(t)}_{\mathbf{H}^{-2}} \lesssim & \norm{X_{PDE}(0) - X_{CS}(0)}_{\mathbf{H}^{-2}} \nonumber\\
    & \quad + \int_0^t   \norm{F(X_{PDE}(s)) - F(X_{CS}(s))}_{\mathbf{H}^{-2}} ds + \int_0^t \norm{R(s)}_{\mathbf{H}^{-2}} ds .
\end{align}
The first component of $F(X_{PDE}(s)) - F(X_{CS}(s))$ is zero, while the second is 
\begin{align}\label{expand_F_second_component}
    & \rho (a*j) - \rho_\epsilon (a*j_\epsilon) - j (a * \rho) + j_{\epsilon} (a * \rho_{\epsilon}) \nonumber\\
    & = \rho a*(j - \overline{v} \rho) - (j - \overline{v} \rho) a * \rho - ( \rho_{\epsilon} a*(j_{\epsilon} - \overline{v} \rho_{\epsilon})- (j_{\epsilon} - \overline{v} \rho_{\epsilon})a * \rho_{\epsilon}  )\nonumber \\
    & = (\rho - \rho_{\epsilon}
    ) a*(j - \overline{v} \rho) + \rho_{\epsilon} a * \left[ (j - \overline{v} \rho) - (j_{\epsilon} - \overline{v} \rho_{\epsilon}) \right] \nonumber\\
    & - \left[ (j - \overline{v} \rho) a* (\rho - \rho_{\epsilon}) + \left[ (j - \overline{v} \rho) - (j_{\epsilon} - \overline{v} \rho_{\epsilon})   \right]  a* \rho_{\epsilon} \right] .
\end{align} 
Using the following inequality valid for $d=1$ \cite[Theorem 8.1]{Behzadan2021}
\begin{align}
    \label{eq:A}
    \norm{uv}_{H^{-2}} \lesssim \norm{u}_{H^{-2}} \norm{v}_{H^{2}}, \quad \forall u\in H^{-2}, v\in H^2,
\end{align}
as well as \eqref{bound_fourier_a}, we deduce that, for general $u,v\in H^{-2}$, we have
    \begin{align}\label{convolutional_inequality_H_1}
        \norm{u( a*v)}_{H^{-2}}^2 & \stackrel{\eqref{eq:A} }{\lesssim}\norm{u}_{H^{-2}}^2 \norm{a*v}_{H^{2}}^2 \lesssim \norm{u}_{H^{-2}}^2 \left[ \sum_{r=0}^{2}\norm{a^{(r)}*v}_{L^2}^2  \right] \nonumber\\
        & \lesssim \norm{u}_{H^{-2}}^2 \left[ \sum_{r=0}^{2}\sum_{\xi}| \widehat{a^{(r)}*v}(\xi)|^2\right] = \norm{u}_{H^{-2}}^2 \left[ \sum_{r=0}^{2}\sum_{\xi}| \widehat{a^{(r)}}(\xi)\widehat{v}(\xi)|^2 \right] \nonumber\\
        & \stackrel{\eqref{bound_fourier_a}}{\lesssim} \norm{u}_{H^{-2}}^2 \left[\sum_{\xi} |\hat{v}(\xi)|^2 (1+ |\xi|^2)^{-2} \right] =  \norm{u}_{H^{-2}}^2 \norm{v}_{H^{-2}}^2 
    \end{align}
If applied to \eqref{expand_F_second_component}, the inequality \eqref{convolutional_inequality_H_1}
grants
\begin{align*}
    \norm{F(X_{PDE}) - F(X_{CS})}_{\mathbf{H}^{-2}} \lesssim & \,C(a)  [ \norm{\rho - \rho_\epsilon}_{H^{-2}} \cdot \norm{j - \overline{v} \rho}_{H^{-2}} \\    
    & + \norm{\rho_\epsilon}_{H^{-2}} \cdot \norm{ (j - \overline{v} \rho) - (j_\epsilon - \overline{v} \rho_\epsilon)}_{H^{-2}} ] . 
\end{align*}
This entails that
%To proceed we can either split, or not split, the $\norm{ (j_1 - \overline{v} \rho_1) - (j_2 - \overline{v} \rho_2)}_{H^{-1}}$ term. In the no split case, we get

\begin{align*}
    \norm{X_{PDE}(t) - X_{CS}(t)}_{\mathbf{H}^{-2}} \leq & \norm{X_{PDE}(0) - X_{CS}(0)}_{\mathbf{H}^{-2}} \\
    & \quad + \int_0^t A(s)  \norm{X_{PDE}(s) - X_{CS}(s)}_{\mathbf{H}^{-2}} ds + \int_0^t B(s) ds ,
\end{align*}
where we have set 
\begin{align*}
    A(s) & := C(a) \cdot \norm{j(s) - \overline{v} \rho(s)}_{H^{-2}} \\
    B(s) & :=  \norm{R(s)}_{H^{-2}} + \norm{\rho_\epsilon}_{H^{-2}} C(a) \norm{ (j(s) - \overline{v} \rho(s)) - (j_\epsilon(s) - \overline{v} \rho_\epsilon(s))}_{H^{-2}} \\
    & \quad \lesssim \norm{ R_{closing}(s)}_{H^{-2}} + \norm{R_{\epsilon}(s)}_{H^{-2}} +\norm{\rho_\epsilon}_{H^{-2}} C(a) \norm{ (j(s) - \overline{v} \rho(s)) - (j_\epsilon(s) - \overline{v} \rho_\epsilon(s))}_{H^{-2}}.
\end{align*}

%In the split case we have 
%\begin{align*}
%    \norm{X_1(t) - X_2(t)}_{\underline{H}^{-1}} \leq & \norm{X_1(0) - X_2(0)}_{\underline{H}^{-1}} + \int_0^t \tilde{A}(s)  \norm{X_1(s) - X_2(s)}_{\underline{H}^{-1}} ds + \int_0^t \tilde{B}(s) ds ,
%\end{align*}

%where 
%\begin{align*}
%    \tilde{A}(s) := & C(a) \cdot \left( \norm{j_1 - \overline{v} \rho_1}_{H^{-1}} + norm{\rho_2}_{H^{-1}} \right) \\
%    \tilde{B}(s) := & \norm{R(s)}_{H^{-1}} .
%\end{align*}

%For now let us consider the non-split case. 
Then, using Gr\"onwall's lemma gives 
\begin{align}
    \label{eq:Split_bound}
    \norm{X_{PDE}(t) - X_{CS}(t)}_{\mathbf{H}^{-2}} & \lesssim \norm{X_{PDE}(0) - X_{CS}(0)}_{\mathbf{H}^{-2}} \exp \left\{ \int_0^t A(s) ds  \right\} \nonumber\\
    & \quad + \int_0^t B(s) \exp \left\{ \int_s^t A(r) dr  \right\} ds .
\end{align}

\emph{Step 3: Concluding the argument.} Using Proposition \ref{local_result}, we immediately get
\begin{align}\label{bound_A_s}
A(s) \leq C(a) \|j_0 - \overline{v} \rho_0\| e^{-C_as}
\end{align}
Bounding $B(s)$ relies on the following: i) the bound $\|R_{closing}\|_{H^{-2}} \lesssim R_v e^{-a(x_M)t}$, which follows from the characterisation of the $H^{-2}$-norm 
$$
\|R_{closing}\|_{H^{-2}} = \sup_{\varphi \in H^2(\mathbb{T})\colon \|\varphi\|_{H^2(\mathbb{T})}=1}{|\int_{\mathbb{T}}{R_{closing} \varphi|}},
$$
as well as the one-dimensional embedding $H^{2}\subset W^{1,\infty}$; ii) the bound $\|R_\epsilon\|_{H^{-2}}\lesssim C(a)R_v e^{-a(x_M)t}\sqrt{\epsilon}$ from Proposition \ref{prop:Mollifying_err}; iii) the simple bound $\|\rho_\epsilon\|_{H^{-2}} \lesssim 1$, which follows from the one-dimensional embedding $H^{1}\subset L^{\infty}$
iv) the bound \eqref{particle_flocking_bound} and Proposition \ref{local_result}, which allows us to deal with the term $\norm{ (j(s) - \overline{v} \rho(s)) - (j_\epsilon(s) - \overline{v} \rho_\epsilon(s))}_{H^{-2}}$. Altogether, we obtain
\begin{align}\label{bound_B_s} 
B(s) \leq  C(a)R_v e^{-a(x_M)s}(1+\sqrt{\epsilon}) + C(a) \left\{ \|j_0 - \overline{v} \rho_0\| e^{-C_as} + R_v e^{-a(x_M)s} \right\}.
\end{align}
The proof is concluded by plugging \eqref{bound_A_s} and \eqref{bound_B_s} in \eqref{eq:Split_bound}.
\end{proof}

%\begin{remark}
%    {\color{magenta} FC: adapt this remark}.  We have considered the $\mathbf{H}^{-1}$-norm instead of the $L^2$-norm because if we consider the closing approximation for the particles 
%\begin{align}
%    \left[ R(s)  \right]_j := - N^{-1} \sum_{i=1}^N (v_i^2(s) - \overline{v}^2) \delta_\epsilon' (x-x_i(s))
%\end{align}
%evaluated in the $L^2$-norm then there do not exist good bounds for $\norm{f(q,p,t)}_{C^1}$, namely only $\norm{f(q,p,t)}_{C^1} \propto e^{\alpha t}$ (see the remarks on p. $14$ in \cite{HaTadmor2008}). If reasonable bounds involving $\partial_q f(q,p,t)$ could be obtained then we could also obtain a bound a for the $L^2$-norm.
%\end{remark}

\section{Numerical results}
\label{sec:Numerical}

In this section we present some numerical results for the prototypical interaction potential \eqref{eq:NonConstComm}, to complement the analytical results presented in the previous sections. Unless otherwise stated we consider $d=1$ and parameters $r = 1/2$, to guarantee unconditional flocking, and $\lambda = 50$, so that flocking occurs relatively quickly. 

\subsection{Accuracy in the $L^2$-norm}
\label{subsec:NumPDEvsABM}

To extend the results of Section \ref{sec:ReductionAccuracy}, on the error between CS particle and PDE systems in the $\mathbf{H}^{-2}$-norm, we present numerical results for the error in the $L^2$-norm in this subsection. Let us recall that the upper bound we derived for the error is dependent on a variety of factors, namely: the spread of the particles initial velocities; the standard deviation of the regularised measures; and the flocking rate of both particle and PDE models. In the $L^2$-norm the error bound follows a similar scaling, however, with an additional contribution from $\partial_x f$ from the term
\begin{align}
    \label{eq:R_L2}
    \mathbb{E} \left[ \norm{R(s)_{closing}}_{L^{2}}^2 \right] & = \mathbb{E} \left[ \norm{ N^{-1} \sum_{i=1}^N (v_i^2(s) - \overline{v}^2) \delta_\epsilon'(x-x_i)}_{L^{2}}^2 \right] ,
\end{align}
where we need to offload the derivative on $\delta_\epsilon$ to $f$.

To investigate the impact of $\partial_x f$, as well as the other factors, on an upper bound for \eqref{eq:R_L2}, we present numerical results, where \eqref{eq:R_L2} is evaluated at one point in time, $t=0$. We first fix $\epsilon =0.1$ and vary the initial velocity spread, by changing the length of the uniform distribution from which the initial velocities are drawn. Furthermore, we also investigate the effects of the initial position distribution becoming more non-uniform, i.e., increasing $\partial_x f$. This is done by randomly generating the positions of the particles from von-Mises distributions, with probability density function
\begin{align*}
    VM(x | k) = \frac{e^{k \cos{(x)}}}{2 \pi I_0(k)} ,
\end{align*}
where $I_0(\cdot)$ is the modified Bessel function of the first kind of order $0$, with varying $k$. As $k$ is increased, the distributions become more localised and less uniform ($k = 0$ corresponds to the uniform distribution). The number of particles is fixed at $N=10^{3}$ and $10^{3}$ realisations are used to achieve reliable statistics. The results for ten different velocity spreads and $k =0 , 0.5, 1, 1.5$ and $2$ are presented in Fig.~\ref{fig:C_dep}(a). From these results, we can clearly see that the profile of the particle position's distribution, has a negligible effect on \eqref{eq:R_L2}, when compared to the impact of the initial velocity spread.

To investigate the effect of $\epsilon$ we fix the velocity spread by drawing initial velocities from a uniform distribution centred at zero with a width of ten and vary $\epsilon$ and $k$. These results are shown in Fig.~\ref{fig:C_dep}(b). For $k = 0$ the closing error clearly follows a power law decay in epsilon with exponent $-3$. To illustrate this, a function proportional to $\epsilon^{-3}$ is drawn in the plot as a solid black line. As the position distribution becomes more localised, the closing error however clearly no longer follows a power law decay. For $k >0$ and larger $\epsilon$ the error is instead dominated by the contribution from the profile of the spatial distribution. Evidently, the higher $k$ and therefore $\partial_x f$, the worse an upper bound for \eqref{eq:R_L2} would be.

\begin{figure}[h]
    \includegraphics[width=0.99\linewidth]{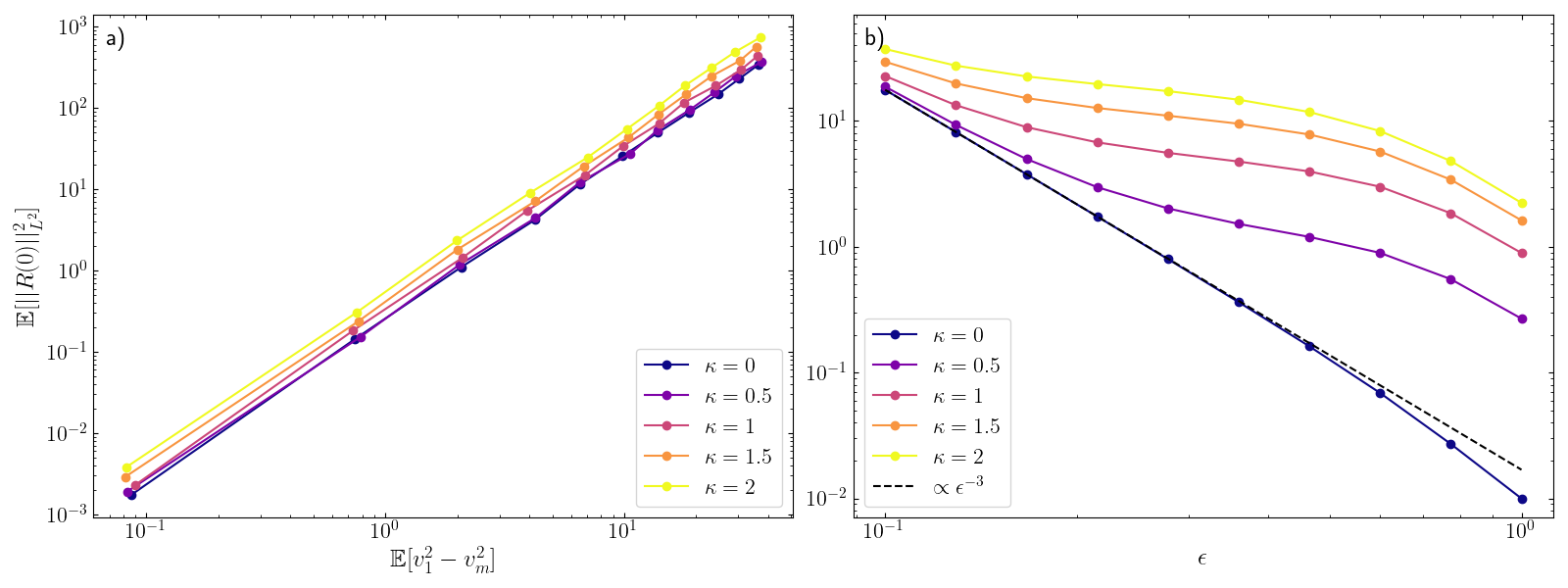}
\caption{The term $\mathbb{E} \left[ \norm{R(s)}_{L^{2}}^2 \right]$, \eqref{eq:R_L2} evaluated at $t=0$ for a variety of parameter settings. a) The dependence of \eqref{eq:R_L2} on the initial velocity spread of the particle system for fixed $\epsilon = 0.1$. The velocity spread is adjusted by changing the width of the uniform distribution from which the initial velocities are drawn. Furthermore, the effect of the shape of the initial position distribution is also shown by varying $\kappa$ of the von-Mises distribution from which initial particles positions are drawn. b) Similar to a) but showing the dependence fot fixed velocity spread, and varying $\epsilon$.} \label{fig:C_dep}
\end{figure}

Let us now investigate how the errors in the $L^2$-norm
\begin{align}
    \label{eq:L2_err}
    \mathbb{E} \left[  \norm{X_{PDE}(t) - X_{CS}(t)}_{L^{2}}\right] ,
\end{align}
compare for the reduced PDE with and without the inclusion of the weight-term, discussed in Section \ref{subsec:Weights}.

Similar to the results in Fig.~\ref{fig:C_dep} the initial velocity spread and $\epsilon$ are varied and initial particle positions are drawn from a uniform distribution. The accuracy of the PDE at $t=2$ (solid black line) is shown in Fig.~\ref{fig:PDE_closing_comp}(a) and (b) for ten different velocity spreads and ten different $\epsilon$'s, respectively. As expected, the further in the flocking regime the particles are (this corresponds to a small velocity spread) the more accurate the PDE is. The accuracy of the PDE does decrease with $\epsilon$, a drawback when a higher resolution of the density profile of the particles is desired.

As discussed in Section \ref{subsec:Weights} choosing an appropriate weight term $w(t)$ to compensate for the closing approximation, by accounting for the deviation from the flocking regime, may give a better approximation to the particle dynamics. We introduced an appropriate choice of weight \ref{exp_weight}, which only depends on the initial closing approximation $w_0(x)$, current time of the system, and rate of flocking $C_a$. For the interaction potential \eqref{eq:NonConstComm}, we set $C_a = \lambda/2$.

In Fig.~\ref{fig:PDE_closing_comp} we plot the accuracy of the PDE with the inclusion of the weight term using dashed, blue lines (all other system parameters as well as the numerical implementation are the same as for the previously considered case of $w=1$). We see that the weight term causes a significant increase in the accuracy (of a factor $\gtrsim 5$) for when the reduced model would typically perform poorly, i.e. high velocity spread or low $\epsilon$. While we do not yet understand the exact effect of the weight term and how best to choose it, these results indicate that it is a modification with great potential to increase the accuracy of the reduced PDE.

\begin{figure}[h]
    \includegraphics[width=0.99\linewidth]{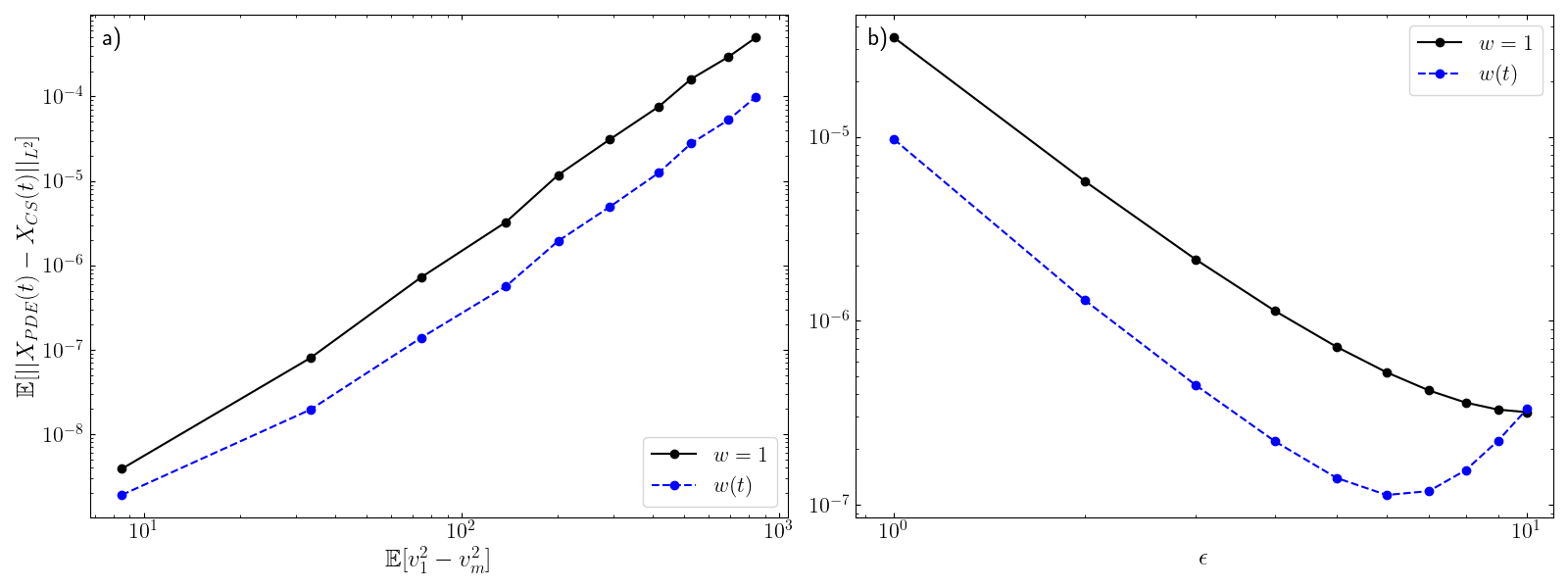}
    \caption{The discrepancy between the particle system and reduced PDE \eqref{eq:L2_err} in one dimension, evaluated at $t=2$. a) The error for fixed $\epsilon$ and varying initial velocity spread with initial positions drawn from a uniform distribution centred at zero with width 40. The solid, black line is for the reduced PDE model without the inclusion of the weight \eqref{reduced_CS_rho}-\eqref{reduced_CS_j} and the dashed, blue line for the PDE \eqref{reduced_CS_rho_w}-\eqref{reduced_CS_j_w} with the inclusion of the weight \eqref{exp_weight}. b) Similar to a), but for fixed initial velocity spread and varying $\epsilon$.}
    \label{fig:PDE_closing_comp}
\end{figure}

\subsection{Comparison with the hydrodynamic model \eqref{eq:Hydrodynamic1}-\eqref{eq:Hydrodynamic2}} \label{subsec:HydroCompare}

In this subsection we compare numerical results of our reduced, inertial model \eqref{reduced_CS_rho}-\eqref{reduced_CS_j} with the previously derived hydrodynamic description \eqref{eq:Hydrodynamic1}-\eqref{eq:Hydrodynamic2}, in one-dimension. The empirical momentum density $j$ and the momentum $\rho u$ are analogous in the two models, which can be easily seen by substituting the empirical density for $x_i$ and $v_i$ into the definition of $\rho u$. The two models will therefore always agree at the initial time if we choose the same $\rho_0$ and $u_0 = j_0/\rho_0$. For initial data, that is in a flock, i.e., $j_0 = \overline{v} \rho_0$, the two reduced PDE models agree for all time, as they both transform into the wave equation. 

In Fig.~\ref{fig:Hydro_comp} a) we present the $L^2$ difference between the two reduced models as a function of time for five different initial data which are progressively further from flocking, i.e., $\norm{j_0 -\overline{v} \rho_0}_{L^2}$ is higher (corresponding to brighter colours in the figure). For all the tested initial data, there is an early increase in the discrepancy between the two models, which then smoothens out for later times. From the work done in Section \ref{sec:PDEflock} as well as the references given in Section \ref{subsec:CSmodel} we know that both of the reduced models will flock. This means that eventually the difference between their results will be constant, which is the flattening effect we observe. We also see that when the initial data is closer to flocking, the models agree more, as they are both closer to behaving like the wave equation. 

Instead of comparing the models to each other, in Fig.~\ref{fig:Hydro_comp} b) we compare both models with the results from the particle system. This is done similarly to the results presented in Section \ref{subsec:NumPDEvsABM}, by taking $10^3$ realisations of the CS model for varying initial velocity spreads and then computing \eqref{eq:L2_err} at $t=5$ for both reduced PDE models. From these results, we can see that our model reproduces the results of the particle system with the same accuracy as the discretization of the hydrodynamic description. While a more detailed, analytical, comparison of these two models is left for future work, the results presented in Fig.~\ref{fig:Hydro_comp} a) and b) demonstrate that our reduced model is not only different from the hydrodynamic description, but also reproduces the dynamics of the CS model as accurately.

\begin{figure}[h]
    \includegraphics[width=0.99\linewidth]{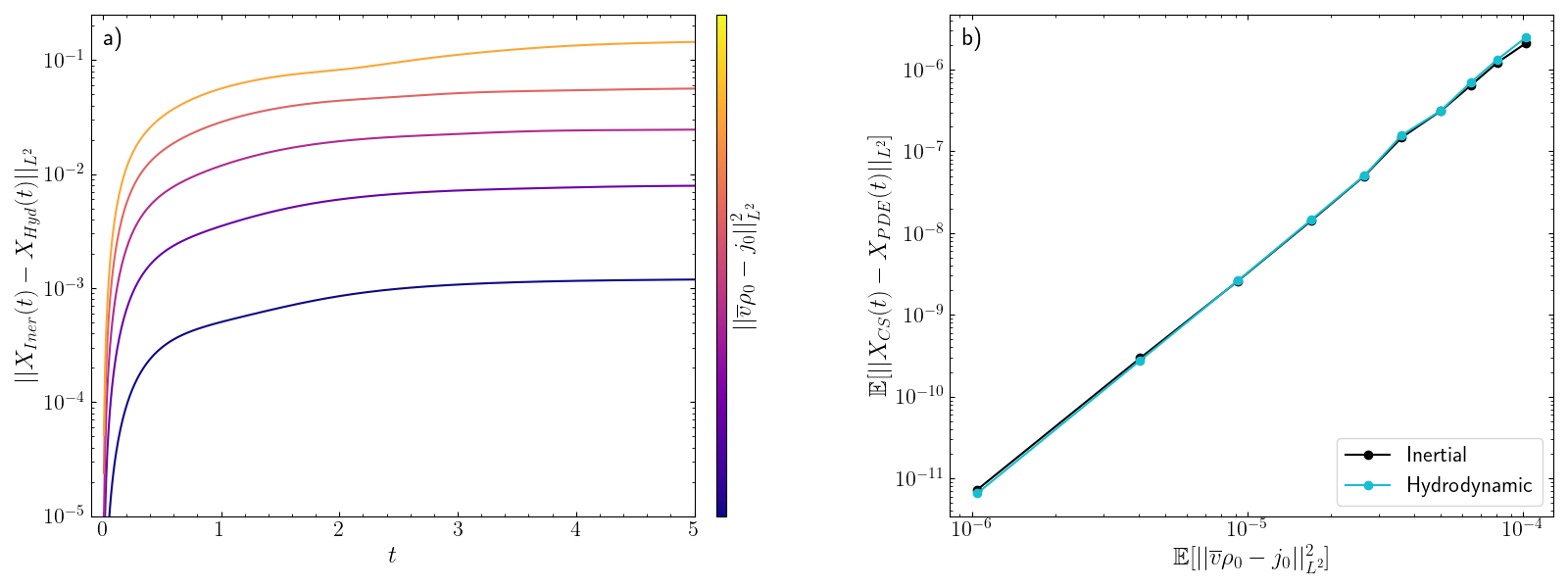}
    \caption{a) The difference between the reduced model \eqref{reduced_CS_rho}-\eqref{reduced_CS_j} and the hydrodynamic model \eqref{eq:Hydrodynamic1}-\eqref{eq:Hydrodynamic2} measured in the $L^2$ norm, as a function of time. Five different initial data $\rho_0$, $j_0$ are taken. The initial data with higher $\norm{j_0 -\overline{v} \rho_0}_{L^2}$ correspond to more brightly coloured lines. b) The $L^2$ error \eqref{eq:L2_err} between the CS model and our reduced PDE model, as well as the error of the hydrodynamic model.}
    \label{fig:Hydro_comp}
\end{figure}

\subsection{Computational effort}
\label{subsec:CompEffort}

The primary motivation for using the reduced inertial PDE model, is that it reduces the state space of the MFL, and, unlike the particle system, the computational effort for simulating the PDE does not increase with the total number of particles, $N$. This is because the reduced model is independent of $N$, while for the Cucker-Smalle model increasing $N$ means that more pairwise distances need to be computed and more ODEs for the positions and velocities of the particles solved. We demonstrate that this is the case by varying $N$ for both models and measuring the time it takes to simulate one time step $\Delta t$. These results are presented in Fig.~\ref{fig:CPU_time} for both the $1$d and $2$d systems on a log-log scale. As expected, the time needed to simulate the PDE is independent of $N$, while the computational effort required to compute the particle system increases super-linearly with $N$. For both the $1$D and $2$D case the reduced PDE model already takes less computational effort for $N\approx 600$ particles.
%\ana{for the submission in the journal we can think if we want to compare computational costs with MFL reduced model, I would in any case really be interested in it, might be one more reason why to use this model}

\begin{figure}[h]
    \includegraphics[width=0.99\linewidth]{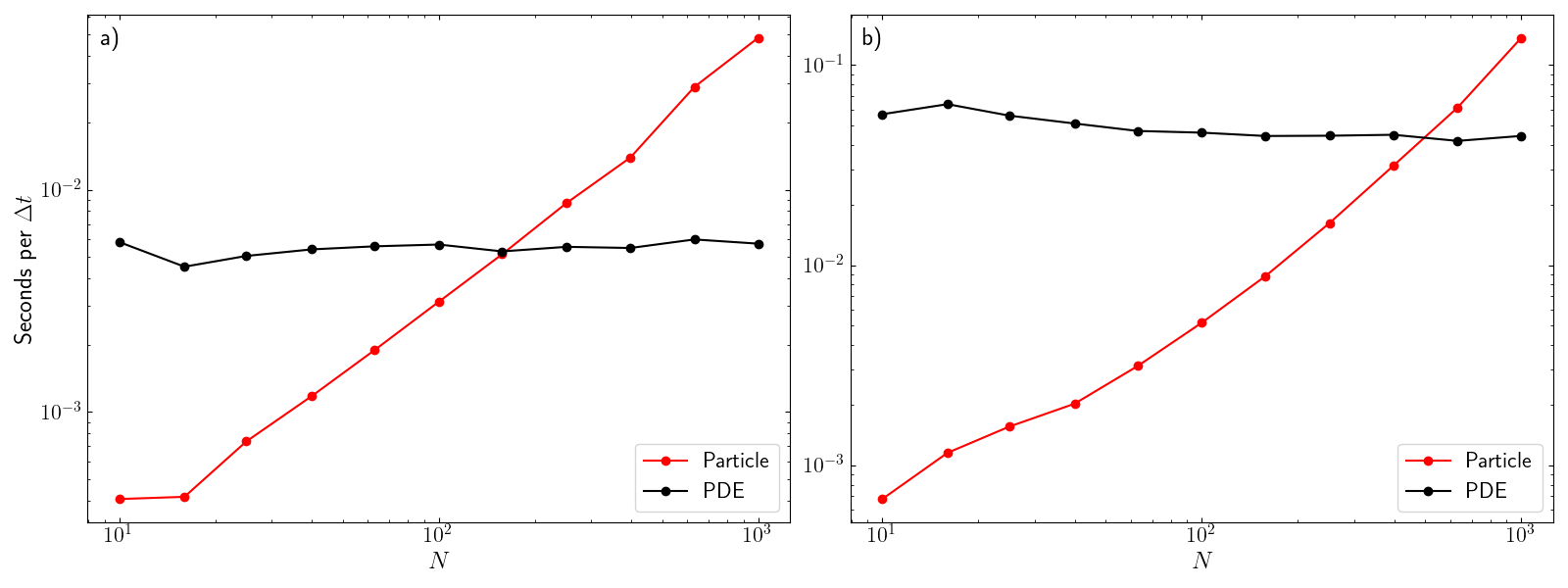}
    \caption{a) The mean CPU time required (in seconds) to compute one time step of the Cucker-Smale model and reduced PDE in dimension $d=1$ for various particle numbers $N$. For each point, the positions and velocities of the particles are randomly generated and the mean computed over $10^3$ of such realisations. b) Similar to a), but for dimension $d=2$.}
    \label{fig:CPU_time}
\end{figure}

\subsection{Stochastic PDE}
\label{subsec:SPDE}

In Section \ref{subsubsec:SPDEderivation} we introduced a stochastic modification of the deterministic Cucker-Smale model \ref{eq:stoc_det_v} and derived a reduced stochastic PDE model \ref{eq:reduced_SCS_j}. To demonstrate that the reduced stochastic model replicates the statistics of the stochastic particle system, we compare the computed mean and variance of $\rho$ and $j$ at a final time $t=0.5$. We take deterministic initial data (in the flocking regime) and perform a total of $10^{3}$ realisations to achieve accurate statistics. These results are presented in Fig.~\ref{fig:Simple_noise}, from which we can see that similar to the deterministic case, the reduced model is able to reproduce the results of the original system.

\begin{figure}[h]
    \includegraphics[width=0.99\linewidth]{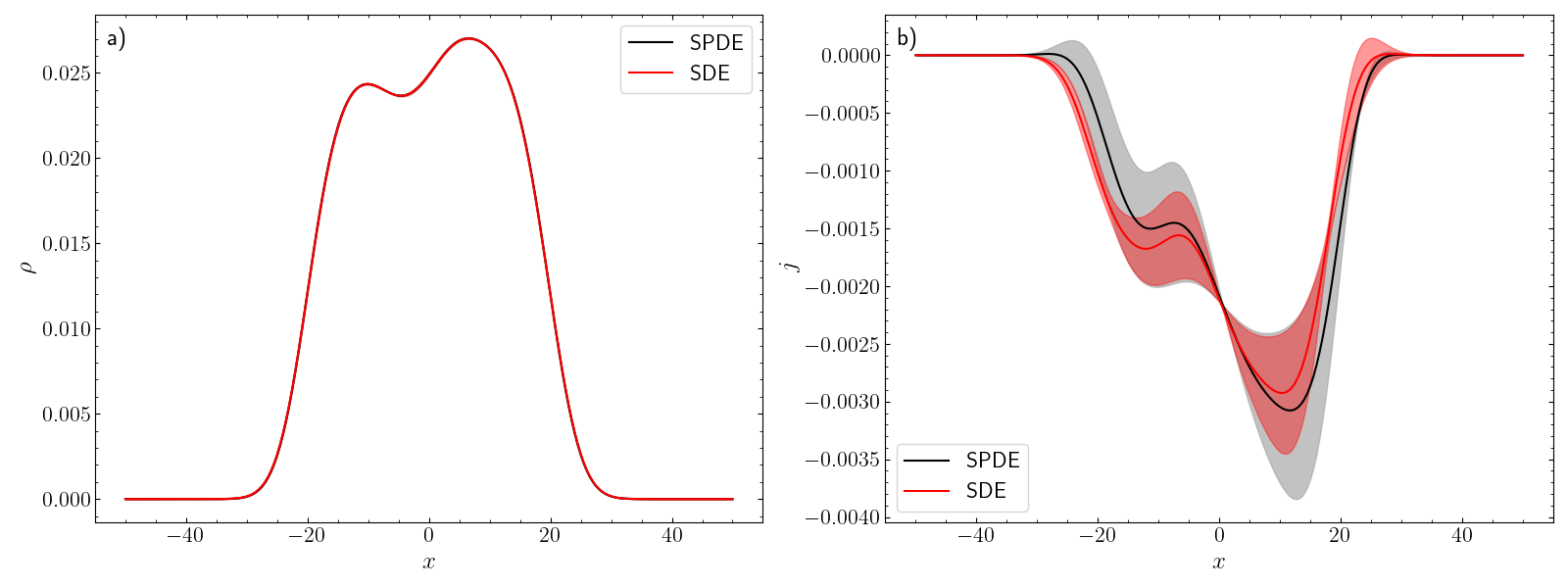}
    \caption{Results obtained for the stochastic Cucker-Smale model \eqref{eq:stoc_det_x}-\eqref{eq:stoc_det_v} and SPDE \eqref{reduced_SCS_rho}-\eqref{eq:reduced_SCS_j} at time $t=0.5$. a) the mean and variance of the empirical density $\rho$ of the particle system (red line and shaded light-red region, respectively) and reduced model (black line and shaded grey region, respectively). A total of $10^3$ realisation are computed to obtain the mean and variance. b) The same as a), but for the empirical momentum density $j$. Note that because the standard deviation for $\rho$ is so low it is not visible in the plot.}
\label{fig:Simple_noise}
\end{figure}

\section{Conclusion}\label{sec:Conclusion}
%Short summary of the paper. Highlight our coolness.\\
%Future topics:
%more detailed (analytical) comparison to the hydrodynamic model will be a topic for the future work.\\
%Further analytical considerations of the associated analysis of fluctuations will be studied in the future.  

In this work we have presented and analysed a reduced, inertial PDE model which is a computationally cheap alternative to the Cucker-Smale particle model.
%, which is used to describe the flocking of particles. 
The reduced model is derived by defining the empirical particle density and momentum density, which are evaluated only in time and the position variable, not the velocity variable. To close our reduced model we rely on the approximation that the velocities of all particles are close to alignment, i.e., the particles are in the flocking regime.  

In our first main result (Theorem \ref{thm:all_PDE_flocking_results}) we prove PDE flocking for our inertial reduced PDE model (in all dimensions, and with a high-density regularisation for $d>1$). Furthermore, in the case $d=1$, we also introduce a modification of the inertial reduced PDE model, by including a weight term which corrects for the closing approximation used, which numerical results indicate drastically decreases the error, measured with the $L^2$-norm, with no extra computational cost. The modified model is shown to preserve the well-posedness and flocking property of the inertial model, but a quantitative bound for the accuracy with which it reproduces the particle dynamics is left for future work. In our second result (Theorem \ref{thm_error_PDE_particles_H_minus_2}), valid in dimension $d=1$, we quantify the error between the particle system and reduced PDE in the space $\mathbf{H}^{-2}$, deriving an upper bound that depends on the initial spread of particles, the flocking rate of the PDE and particle systems, and the smoothing of the initial empirical densities. 

Our model is an alternative to the well known hydrodynamic description of the CS model, which also is not evaluated in the velocity variable and relies on an appropriate closing approximation. Using numerical results we have demonstrated that the inertial model is: i) different from the hydrodynamic model, and ii) as accurate in replicating the particle system. An analytical comparison between these two models is a topic for future work. The key difference between the inertial, reduced model and hydrodynamic model, is that our model is derived from only the particle system, while the hydrodynamic model is obtained from the mean-field limit. We consider our alternative derivation of a reduced model bypassing the mean-field limit as advantageous, since it allows for a direct connection to the underlying particle system, and for the analysis of fluctuations in the velocity field. This work lays the foundations for the quantitative analysis of fluctuations, which we will investigate in subsequent works. \\

\enlargethispage{20pt}

%\ack{The authors would like to thank ... \natasa{Who do we want to thank???} \ana{Johanes}}

{\bfseries Funding.} This work has been partially funded by the Deutsche Forschungsgemeinschaft (DFG) under Germany’s Excellence Strategy through grant EXC-2046 The Berlin Mathematics Research Center MATH+ (project no. 390685689). ADj gratefully acknowledges funding by Daimler and Benz Foundation as part of the scholarship program for junior professors and postdoctoral researchers. 

%%%%%%%%%% Insert bibliography here %%%%%%%%%%%%%%

\vskip2pc

\bibliographystyle{abbrv}
\bibliography{references}

\begin{thebibliography}{10}

\bibitem{Ahn2012}
S.~Ahn, H.~Choi, S.~Ha, and H.~Lee.
\newblock On collision-avoiding initial configurations to {C}ucker-{S}male type
  flocking models.
\newblock {\em Communications in Mathematical Sciences}, 10(2):625--643, 2012.

\bibitem{Behzadan2021}
A.~Behzadan and M.~Holst.
\newblock Multiplication in sobolev spaces, revisited.
\newblock {\em Arkiv för Matematik}, 59, 2021.

\bibitem{CarrilloEtAl2010A}
J.~A. Carrillo, M.~Fornasier, J.~Rosado, and G.~Toscani.
\newblock Asymptotic flocking dynamics for the kinetic cucker–smale model.
\newblock {\em SIAM Journal on Mathematical Analysis}, 42(1):218--236, 2010.

\bibitem{CarrilloEtAl2010B}
J.~A. Carrillo, M.~Fornasier, G.~Toscani, and F.~Vecil.
\newblock {\em Particle, kinetic, and hydrodynamic models of swarming}, pages
  297--336.
\newblock Birkh{\"a}user Boston, Boston, 2010.

\bibitem{10.1214/18-AAP1400}
P.~Cattiaux, F.~Delebecque, and L.~P{\'e}d{\`e}ches.
\newblock {Stochastic Cucker–Smale models: Old and new}.
\newblock {\em The Annals of Applied Probability}, 28(5):3239 -- 3286, 2018.

\bibitem{Choi2017}
Y.-P. Choi, S.-Y. Ha, and Z.~Li.
\newblock {\em Emergent Dynamics of the Cucker--Smale Flocking Model and Its
  Variants}, pages 299--331.
\newblock Springer International Publishing, Cham, 2017.

\bibitem{Cornalba2019}
F.~Cornalba, T.~Shardlow, and J.~Zimmer.
\newblock A regularized dean--kawasaki model: Derivation and analysis.
\newblock {\em SIAM Journal on Mathematical Analysis}, 51(2):1137--1187, 2019.

\bibitem{cornalba2020weakly}
F.~Cornalba, T.~Shardlow, and J.~Zimmer.
\newblock From weakly interacting particles to a regularised dean--kawasaki
  model.
\newblock {\em Nonlinearity}, 33(2):864, 2020.

\bibitem{cornalba2021well}
F.~Cornalba, T.~Shardlow, and J.~Zimmer.
\newblock Well-posedness for a regularised inertial dean--kawasaki model for
  slender particles in several space dimensions.
\newblock {\em Journal of Differential Equations}, 284:253--283, 2021.

\bibitem{Couzin2003}
I.~D. Couzin and J.~Krause.
\newblock Self-organization and collective behavior in vertebrates.
\newblock {\em Advances in The Study of Behavior}, 32:1--75, 2003.

\bibitem{cucker2007emergent}
F.~Cucker and S.~Smale.
\newblock Emergent behavior in flocks.
\newblock {\em IEEE Transactions on automatic control}, 52(5):852--862, 2007.

\bibitem{Cucker2007b}
F.~Cucker and S.~Smale.
\newblock On the mathematics of emergence.
\newblock {\em Japanese Journal of Mathematics}, 2:197--227, 2007.

\bibitem{Dean1996}
D.~Dean.
\newblock {L}angevin equation for the density of a system of interacting
  {L}angevin processes.
\newblock {\em J. Phys. A}, 29:L613--L617, 1996.

\bibitem{Dragomir2003}
S.~S. Dragomir.
\newblock Some gronwall type inequalities and applications.
\newblock {\em Mathematics eJournal}, 2003.

\bibitem{HaLiu2009}
S.~Ha and J.~Liu.
\newblock A simple proof of the {C}ucker-{S}male flocking dynamics and
  mean-field limit.
\newblock {\em Communications in Mathematical Sciences}, 7:297--325, 2009.

\bibitem{HaKang2014}
S.~Y. Ha, M.-J. Kang, and B.~Kwon.
\newblock A hydrodynamic model for the interaction of {C}ucker-{S}male
  particles and incompressible fluid.
\newblock {\em Mathematical Models and Methods in Applied Sciences}, 24, 10
  2014.

\bibitem{HaKang2015}
S.-Y. Ha, M.-J. Kang, and B.~Kwon.
\newblock Emergent dynamics for the hydrodynamic {C}ucker--{S}male system in a
  moving domain.
\newblock {\em SIAM Journal on Mathematical Analysis}, 47(5):3813--3831, 2015.

\bibitem{HaTadmor2008}
S.-Y. Ha and E.~Tadmor.
\newblock From particle to kinetic and hydrodynamic descriptions of flocking.
\newblock {\em Kinetic and Related Models}, 1(3):415--435, 2008.

\bibitem{Helfmann2023}
L.~Helfmann, N.~D. Conrad, P.~Lorenz-Spreen, and C.~Sch{\"u}tte.
\newblock Modelling opinion dynamics under the impact of influencer and media
  strategies.
\newblock {\em Scientific Reports}, 13:19375, 2023.

\bibitem{Karper2013}
T.~K. Karper, A.~Mellet, and K.~Trivisa.
\newblock Existence of weak solutions to kinetic flocking models.
\newblock {\em SIAM Journal on Mathematical Analysis}, 45(1):215--243, 2013.

\bibitem{Karper2015}
T.~K. Karper, A.~Mellet, and K.~Trivisa.
\newblock Hydrodynamic limit of the kinetic {C}ucker–{S}male flocking model.
\newblock {\em Mathematical Models and Methods in Applied Sciences},
  25(01):131--163, 2015.

\bibitem{mitronovic1991inequalities}
D.~Mitronovic, J.~Pecaric, and A.~Fink.
\newblock Inequalities involving functions and their integrals and derivatives,
  1991.

\bibitem{Motsch2011}
S.~Motsch and E.~Tadmor.
\newblock A new model for self-organized dynamics and its flocking behavior.
\newblock {\em Journal of Statistical Physics}, 144:923--947, 2011.

\bibitem{nguyen2022propagation}
V.~Nguyen and R.~Shvydkoy.
\newblock Propagation of chaos for the cucker-smale systems under heavy tail
  communication.
\newblock {\em Communications in Partial Differential Equations},
  47(9):1883--1906, 2022.

\bibitem{Tadmor2014}
E.~Tadmor and C.~Tan.
\newblock Critical thresholds in flocking hydrodynamics with non-local
  alignment.
\newblock {\em Philosophical Transactions of the Royal Society A: Mathematical,
  Physical and Engineering Sciences}, 372, 2014.

\bibitem{Vicsek2012}
T.~Vicsek and A.~Zafeiris.
\newblock Collective motion.
\newblock {\em Physics Reports}, 517(3):71--140, 2012.
\newblock Collective motion.

\end{thebibliography}

\section{Appendix}
\label{sec:Appendix}

\subsection {Full details for the proof of Proposition \ref{local_result}}\label{Appendix:local_result}
Here we prove in detail the steps I), II), III) highlighted in the \emph{Sketch of proof of Proposition \ref{local_result}}. Without loss of generality, we consider the case of $\overline{v}>0$ (see Remark 
\ref{shifting_vs}).

\emph{Proving I)}. The local existence of a $L^2 \times L^2$-valued solution is a consequence of the mild solution theory for the $C_0$-semigroup generated by the operator $(\rho,j)\mapsto (-\nabla \cdot j,-\overline{v}^2\nabla \rho)$, in $H^{1}\times H^{1}\subset L^2\times L^2$ (see \cite{cornalba2021well} for more extensive details), as well as of the local-Lipschitz nature of the nonlinear component $\rho(a\ast j) - j(a\ast \rho)$.

\emph{Proving II)}.
The flocking property is shown analogously to the proof of Theorem \ref{thm_flocking_alignment}, with the extra term 
\begin{equation}\label{bound_eps_bit}
    \theta \langle \rho (g\ast (j-\overline{v}\rho)) , (j-\overline{v}\rho)) \rangle - \theta \langle ( (j-\overline{v}\rho)g*\rho, j-\overline{v}\rho) \rangle 
     \leq 2\theta\|\rho\| \|g\|\|j-\overline{v}\rho\|^2
\end{equation}
being added to the differential of $\frac{1}{2}\partial_t\| j - \overline{v}\rho \|^2$. Point II) then follows from combining \eqref{stopping_time}, \eqref{bound_eps_bit} and \eqref{diff_square} and get overall
\begin{align}
    \label{flocking_alignment_nonlinear}
    \|j(t) - \overline{v} \rho(t)\|^2 \leq \exp\{-C_a t\}\|\rho_0-\overline{v}j_0\|^2.
\end{align}

\emph{Proving III)}. Taking the time derivative of the quantity $\eta(t) := \overline{v}^2\|\rho(t)\|^2 + \|j(t)\|^2$ and using integration by parts, we get
\begin{align}\label{differential_eta}
\frac{1}{2}\partial_t\eta(t) = \langle \rho(a\ast j) - j(a\ast \rho),j \rangle .
\end{align}
Adding and subtracting $\overline{v}\rho$ in suitable $j$-terms and using the fact that $\int{(j-\overline{v}\rho)}=0$, we obtain
\begin{align*}
\frac{1}{2}\partial_t\eta(t) & = \langle \rho(a\ast (j-\overline{v}\rho)) - (j-\overline{v}\rho)(a\ast \rho),j \rangle \\
& = \langle  \rho (\theta g)\ast (j-\overline{v}\rho), j \rangle - C\|j-\overline{v}\rho\|^2 - \langle j-\overline{v}\rho, C_a\overline{v}\rho\rangle \\
& \quad - \langle[j-\overline{v}\rho](\theta g) * \rho, j-\overline{v}\rho\rangle + \langle[j-\overline{v}\rho](\theta g) * \rho, \overline{v}\rho\rangle =: \sum_{i=1}^{5}{B_i}.
\end{align*}
By definition of $\tau$, we have $\|\rho(t)\|\leq \frac{C_a}{4\theta\|g\|_\infty}$ for $t\leq \tau$, and therefore 
\begin{align*}
|B_1| & \leq \frac{C_a}{4}\|j-\overline{v}\rho\|\|j\| \leq \frac{C_a}{4}\|j-\overline{v}\rho\|\sqrt{\eta}, \\
B_2 + B_4  &\leq 0, \\
|B_3| & \leq C_a\|j-\overline{v}\rho\|\|\overline{v}\rho\| \leq C_a\|j-\overline{v}\rho\|\sqrt{\eta},\\
|B_5| & \leq \frac{C_a}{4}\|j-\overline{v}\rho\|\|\overline{v}\rho\| \leq \frac{C_a}{4}\|j-\overline{v}\rho\|\sqrt{\eta}.
\end{align*}
Overall, we get 
\begin{align}\label{bound_partial_t_eta}
\frac{1}{2}\partial_t\eta(t) \leq |B_1|+|B_3|+|B_5|. 
\end{align}

Integrating \eqref{bound_partial_t_eta} in time, and using the flocking property,  due the restriction $t\leq \tau$, from Point II), we get 
\begin{align*}
\eta(t) & \leq \eta(0) + 3\int_{0}^{t}{C_a\|j(s)-\overline{v}\rho(s)\|\sqrt{\eta(s)}ds} \leq \eta(0) + \int_{0}^{t}{3C_a\|j_0-\overline{v}\rho_0\|e^{-(C_a/2)s}\sqrt{\eta(s)}ds}.
\end{align*}
Using a nonlinear generalisation of Gr\"onwall's inequality \cite[Theorem 21]{Dragomir2003} we obtain
\begin{align*}
\eta(t) & \leq \left\{\sqrt{\eta(0)} + 3\int_{0}^{t}{(C_a/2)e^{-(C_a/2)s}ds} \| j_0 - \overline{v} \rho_0 \|\right\}^2 \leq 2\eta(0) + 2\cdot 3^2\|j_0 - \overline{v}\rho_0\|^2.
\end{align*}
This implies that, for $t\leq \tau$
\begin{align*}
\|\rho(t)\| \leq \left( 2\frac{\overline{v}^2\|\rho_0\|^2 + \|j_0\|^2}{\overline{v}^2} + \frac{2\cdot 3^2\|j_0 - \overline{v}\rho_0\|^2}{\overline{v}^2} \right)^{1/2} =: K,
\end{align*}
which completes the proof.

\subsection{Proof of Proposition \ref{local_result_w}}\label{Appendix:weights_flocking}

The proof of Proposition \ref{local_result_w} is, conceptually, identical to that of Proposition \ref{local_result}. The proof of Proposition \ref{local_result} only involves $L^2$ bounds for the unknown densities $\rho,j$, since all higher derivatives of said quantities vanish using integration by parts.
Here, however, due to the presence of the weight $w$, we must make sure that no computation introduces derivatives of the densities $\rho,j$ when integrating by parts: said differently, we are only allowed to offload derivatives onto the weight $w$ if we want to reproduce the proof of Proposition \ref{local_result}. 
\begin{proof}[Proof of Proposition \ref{local_result_w}]
We need to readapt Points I), II), III) of the proof of Proposition \ref{local_result}. 

\emph{Proving I)}. This is done in analogy to Point I) in the proof of Proposition \ref{local_result}, since the weight $w$ is smooth and uniformly bounded away from zero and from above.

\emph{Proving II)}. Unlike what happens in the corresponding step of Proposition \ref{local_result}, we can not directly work with the differentials of the flocking quantity $\|j-\overline{v}\rho\|$ (or equivalently $\|j-\overline{v}w\rho\|$), as to do so 
would introduce higher derivatives of $\rho$, $j$ in to the estimates. The strategy around this problem is split into two sub-steps.
\begin{itemize}
\item \emph{Step IIa)} We show a decaying bound for the quantity 
\begin{align}
R := \langle j-\overline{v}\rho, j-\overline{v}w\rho\rangle = \|j-\overline{v}\rho\|^2 + \langle j-\overline{v}\rho, (1-w)\overline{v}\rho\rangle.
\end{align} 
\item \emph{Step IIb)} We derive the estimate for our standard flocking quantity $\|j-\overline{v}\rho\|$ as a `small' correction (in the $L^2$ sense) of the estimate for $R$ proved in \emph{Step IIa}).
\end{itemize}

\emph{Proving IIa)} The time differential of $R$ is given by 
\begin{align}
\partial_t R 
%& = \langle \partial_t (j-\overline{v}\rho),j-\overline{v}w\rho \rangle + \langle \partial_t (j-\overline{v}w\rho),j-\overline{v}\rho \rangle \nonumber\\
& = \langle \rho (a\ast j)  - j (a \ast \rho) , j-\overline{v}w\rho \rangle + \langle \rho (a\ast j)  - j (a \ast \rho) , j-\overline{v}\rho \rangle \nonumber\\
& \quad - \langle \overline{v}\partial_t w \rho, j-\overline{v}\rho\rangle - \overline{v}\left( \langle \nabla w, j^2/2\rangle - \overline{v}^2\langle \nabla w,\rho^2/2 \rangle \right) := \sum_{i=1}^{4}{Y_i}
\end{align}
where we have used integration by parts in the final equality.
We estimate $Y_1,\dots,Y_4$ individually. First, we define the stopping time 
\begin{align}\label{stop_time_new}
\tau := \inf\left\{t>0: \sqrt{\overline{v}^2\|\rho(t)\|^2 + \|j(t)\|^2} > \overline{v} Q\right\},
\end{align}
for some fixed $Q < C_a/(8\theta\|g\|)$. Therefore, for $t\leq \tau$, we get 
\begin{align}\label{Y3Y4}
|Y_3| \leq 2\overline{v}^2 Q^2 \|\partial_t w\|_{\infty}, \qquad 
|Y_4| \leq \frac{\overline{v}^3}{2} Q^2 \|\nabla w\|_{\infty}.
\end{align}
%{\color{red}ADJ: here I again have a problem, first should it be absolute values of $Y_3$ and $Y_4$ and I dont understand what happened with the mixed term $\overline{v}\| \partial_t w\|_\infty \|j\| \|\rho \|$ }
Since $Q < C_a/{8\theta\|g\|_\infty}$, using the writing $a = C_a  + \theta g$ for the interaction function, adding and subtracting $\overline{v}\rho$, and reusing the bound \eqref{bound_eps_bit}, we get
\begin{align}\label{Y1_Y2}
Y_1 + Y_2 %& = \langle \rho (a\ast j)  - j (a \ast \rho) , j-\overline{v}w\rho \rangle \nonumber\\
%& \quad + \langle \rho (a\ast j)  - j (a \ast \rho) , j-\overline{v}\rho \rangle \nonumber\\
%& = \langle \rho (C_a\overline{v} +(\theta g) \ast j)  - j (C_a + (\theta g)\ast \rho) , j-\overline{v}w\rho \rangle \nonumber\\
%& \quad \quad + \langle \rho (C_a\overline{v} +(\theta g) \ast j)  - j (C_a + (\theta g)\ast \rho) , j-\overline{v}\rho \rangle \nonumber\\
& \leq -C_aR + 2\langle \rho (\theta g)\ast (j-\overline{v}\rho) - (j-\overline{v}\rho) (\theta g)\ast \rho, j-\overline{v}\rho \rangle \nonumber \\
& \quad \quad + \langle \rho (\theta g)\ast j - j (\theta g)\ast \rho, (1-w)\overline{v}\rho \rangle -C_a\|j-\overline{v}\rho\|^2   \nonumber \\
& \leq -C_aR -(C_a/2)\|j-\overline{v}\rho\|^2 + \overline{v}^2\|\theta g\|Q^3\|1-w\|_{\infty}
\end{align}
Combining \eqref{Y1_Y2} and \eqref{Y3Y4} gives
\begin{align}\label{differential_for_R}
\partial_t R \leq -C_aR + \tilde{\beta}(t),
\end{align}	
where we have abbreviated
\begin{align}\label{abbreviation_for_beta_tilde}
\tilde{\beta} := 2\overline{v}^2 Q^2 \left(\overline{v}\|\nabla w\|_{\infty} + \|\partial_t w\|_{\infty} \right) + \overline{v}^2\|\theta g\|Q^3\|1-w\|_{\infty}.
\end{align}
An application of Gr\"onwall's inequality then gives
\begin{align}\label{decaying_bound_for_R}
R(t) \leq R_0 e^{-C_at} + e^{-C_at}\int_{0}^{t}{e^{C_as} \tilde{\beta}(s) ds}.
\end{align}

\emph{Proving IIb)}.
%\federico{For @Ana: constant $1/2$ corrected} 
By noticing that 
$\|j-\overline{v}\rho\|^2 \leq 2R + \overline{v}^2\|w-1\|^2\|\rho\|^2 \leq 2R + \overline{v}^2\|w-1\|^2Q^2$, we deduce from \eqref{decaying_bound_for_R} that
\begin{align}\label{decaying_bound_for_flocking}
\|j(t)-\overline{v}\rho(t)\|^2 & \leq 2R_0 e^{-C_at} + \beta(t) \leq 2|R_0| e^{-C_at} + \beta(t),
\end{align}
where we have set
\begin{align}\label{abbreviation_for_beta}
\beta(t) := 2e^{-C_at}\left(\int_{0}^{t}{e^{C_as} \tilde{\beta}(s) ds}\right) + \overline{v}^2\|w-1\|^2Q^2.
\end{align}
This concludes proving Point II).

\emph{Proving III)} Once again, in order not to introduce derivatives of $\rho$ and $j$, we work with the energy estimates by differentiating the quantity $\eta_w := \overline{v}^2\int{w\rho^2}+\|j\|^2$ rather than $\eta := \overline{v}^2\|\rho\|^2 + \|j\|^2$. Note that, in any case, $\eta_w \propto \eta $ since $w$ is an admissible weight. 

For $t\leq \tau$, we get also using integration by parts
\begin{align}\label{differential_eta_w}
\frac{1}{2}\partial_t \eta_w & = \frac{1}{2}\overline{v}^2\langle (\partial_t w) \rho ,\rho \rangle - \overline{v}^2 \langle w\rho, \nabla \cdot j\rangle + \langle \rho(a\ast j) - j(a\ast \rho),j \rangle - \overline{v}^2 \langle\nabla(w\rho), j\rangle \nonumber\\
& =  \frac{1}{2}\overline{v}^2\langle (\partial_t w) \rho ,\rho \rangle + \langle \rho(a\ast j) - j(a\ast \rho),j \rangle. 
\end{align}
By re-using the same analysis done for \eqref{differential_eta}, the fact that $\eta_w \propto \eta$ (see Definition \ref{defn_admissible_weights}), and the flocking bound \eqref{decaying_bound_for_flocking}, we obtain 
\begin{align}\label{integral_eqn_eta_w}
\eta(t) & \leq \left(\frac{w_{\max}}{w_{\min}}\eta(0) + 3\int_{0}^{t}{C_a\|j(s)-\overline{v}\rho(s)\|\frac{1}{w_{\min}}\sqrt{\eta(s)}ds} + \int_{0}^{t}{\|\partial_t w\|_\infty \frac{1}{2w_{\min}}\eta(s)ds}\right) \nonumber\\
& \leq  \left(\frac{\eta(0)}{w_{\min}} + 3\int_{0}^{t}{\frac{C_a}{w_{\min}}
\left[ 2\sqrt{|R_0|} e^{-(C_a/2)s} + \sqrt{\beta(s)}
 \right]
\sqrt{\eta(s)}ds} + \int_{0}^{t}{\frac{\|\partial_t w\|_\infty}{2w_{\min}}\eta(s)ds}\right).
\end{align}
%\ana{@Federico can you shorten these inequalities, maybe we just keep 8.17 and 8.18 and remove the first steps?}
Using Gr\"onwall's inequality as in \cite[p. 361]{mitronovic1991inequalities}, we deduce
\begin{align}\label{energy_estimate_eta_w}
\eta(t) & \leq \left( \sqrt{\frac{w_{\max}}{w_{\min}}\eta(0)}\exp\left\{\int_{0}^{\infty}{ \|\partial_t w(s)\|_{\infty}\frac{1}{4w_{\min}} ds} \right\}  \right. \nonumber\\
& \quad \left. + \frac{1}{2w_{\min}}\int_{0}^{t}{3C_a
\left[ 2\sqrt{|R_0|} e^{-(C_a/2)s} + \sqrt{\beta(s)}
 \right]}\exp\left\{ \int_{0}^{\infty}{ \|\partial_t w(s)\frac{1}{4w_{\min}}\|_{\infty}dr} \right\} \right)^2 ds \nonumber\\
 & \leq 2\frac{w_{\max}}{w_{\min}}\eta(0)\exp\left\{\int_{0}^{\infty}{ \|\partial_t w(s)\|\frac{1}{2w_{\min}} ds }\right\} \nonumber\\
 & \quad + \frac{ 2\cdot 6^2 }{w_{\min}^2} \left( \sqrt{|R_0|} + \int_{0}^{t}{\sqrt{\beta(s)}ds} \right)^2 \exp\left\{ \int_{0}^{\infty}{ \|\partial_t w(r)\|_{\infty}\frac{1}{2w_{\min}}dr} \right\}
 =: K.
\end{align}
Since $K < \overline{v}^2Q^2$ by assumption, the proof is concluded.
\end{proof}

\subsection{Proof of Lemma \ref{well_posedness_reg}}
\label{app:multi_d_flocking}

%\begin{center}
%\federico{$C_a$ is the constant from the interaction function $a=C_a+\theta g$, $\tilde{C}$ is a constant depending on dimension $d$, $C_P$ is the Poincar\'e constant}

%\federico{I am fine with the entire proof. The comments throughout are improvement one could do, not structural things that should necessarily be changed in a first iteration}
%\end{center}

As in Lemma \ref{flocking_property_multi_d}, set $\tau := \inf\{t>0: \|\rho(\cdot,t)\| \geq C_a/(4\theta\|g\|_\infty)\}$.
The time differential of the quantity $\eta(t) := |\overline{\f{v}}|^2\|\rho(t)\|^2 + \|\f{j}(t)\|^2$ gives
\begin{align*}
\frac{1}{2}\partial_t \eta% & = %|\overline{\f{v}}|^2\langle -\nabla \cdot \f{j},\rho \rangle - \sum_{z=1}^{d}{\sum_{c \in \mathcal{C}_{h,d}}{\varphi(\|(\rho,\f{j})\|_{\nabla,c})\int{e_cv_z^2|\nabla\rho|^2}}} \\
%& \quad + \sum_{z=1}^{d}{\langle j_z (a\ast \rho) - \rho(a\ast j_z), j_z \rangle} - \sum_{z=1}^{d}{\langle \overline{v}_z\nabla \cdot \f{j}, j_z\rangle}  \\
%& \quad - \sum_{z=1}^{d}{\sum_{c \in \mathcal{C}_{h,d}}{\varphi(\|(\rho,\f{j})\|_{\nabla,c})\int{e_c|\nabla j_z|^2}}} \\
& = |\overline{\f{v}}|^2\langle -\nabla \cdot \f{j},\rho \rangle 
+ \sum_{z=1}^{d}{\langle j_z (a\ast \rho) - \rho(a\ast j_z), j_z \rangle} - \sum_{z=1}^{d}{\langle \overline{v}_z\nabla \cdot \f{j}, j_z\rangle}  \\
& \quad - \sum_{z=1}^{d}{\sum_{c \in \mathcal{C}_{h,d}}{\varphi(\|(\rho,\f{j})\|_{\nabla,c})\int{e_c\{|\nabla j_z|^2 + |\overline{v}_z|^2|\nabla\rho|^2\}}}} =: \sum_{i=1}^4T_i. 
\end{align*}
Throughout, we use the Poincar\'e inequality on the whole domain, namely
\begin{align}\label{Poincare_whole_domain}
\|f\| \leq C_P\|\nabla f\| + C_P \left|\int{f}\right|.
\end{align}
where, in order not burden the nature of the constants in subsequent computations, we have combined all constants into just $C_P$.  
%\ana{I dont understand the last term}
%\federico{@Ana: if I use $ \| f - |T|^{-1} \int f \| \leq C_P \| \nabla f \| $ the notation gets significantly messier. I suggest we keep it this way.}
Since $\rho$ and $\{j_z\}_{z=1}^{d}$ preserve their mass %(i.e., $\int{\rho}=1$, $\int{j_z}=\overline{v}_z$, $z=1,\dots,d$, see Subsection \ref{subsec:choice_intial_cond}), 
and Cauchy-Schwartz inequality, we get the bounds
\begin{align}\label{T1}
T_1 & \leq |\overline{\f{v}}|\|\nabla\cdot \f{j}\||\overline{\f{v}}|\|\rho\| \leq \frac{1}{2}|\overline{\f{v}}| \left(\|\nabla\cdot\f{j}\|^2 + |\overline{\f{v}}|^2 \|\rho\|^2\right) \nonumber\\
& \leq \frac{d}{2}|\overline{\f{v}}| \left(|\overline{\f{v}}|^2 \|\rho\|^2 + \sum_{z=1}^{d}{\|\nabla j_z \|^2}\right) \leq  \frac{d}{2}|\overline{\f{v}}| \left(|\overline{\f{v}}|^2 \left[C_P^2\|\nabla\rho\|^2 + C_P^2\right] 
+ \sum_{z=1}^{d}{\|\nabla j_z \|^2}\right) \nonumber\\
& \leq \tilde{C}(d)(C_P^2+1)|\overline{\f{v}}|\left(|\overline{\f{v}}|^2\|\nabla\rho\|^2 + \sum_{z=1}^{d}{\|\nabla j_z \|^2}\right) + \tilde{C}(d)|\overline{\f{v}}|^3C_P^2.
\end{align} 
and, using similar considerations, we also get (possibly for a different $\tilde{C}(d)$)
\begin{align}\label{T13}
T_3 \leq \tilde{C}(d)(C_P^2+1)|\overline{\f{v}}|\left(|\overline{\f{v}}|^2\|\nabla\rho\|^2 + \sum_{z=1}^{d}{\|\nabla j_z \|^2}\right) + \tilde{C}(d)|\overline{\f{v}}|^3C_P^2.
\end{align}
The same computations as in \eqref{bound_partial_t_eta}, together with Poincar\'e and Cauchy-Schwartz inequalities, %and assuming $t\leq \tau$ \federico{mention flocking}, 
give us 
%\federico{yes, currently suboptimal as does not use flocking}
\begin{align}\label{T2}
T_2 \leq C_a\sum_{z=1}^{d}{\|j_z - \overline{v}_z\rho\|\sqrt{\eta}} \leq C_a\tilde{C}(d)(C^2_{P} + 1) \left(|\overline{\f{v}}|^2\|\nabla \rho\|^2 + \sum_{z=1}^{d}{\|\nabla j_z\|^2}\right) + C_a\tilde{C}(d)C^2_{P}|\overline{\f{v}}|^2 
\end{align}
%\sebastian{\eqref{T2} is suboptimal: if we use the flocking up to $\tau$, then $$T_2\leq \|\rho\|\|j_{z,0} - v_z\rho_0\|\sqrt{\eta} \lesssim |\overline{v}|^{-1}\|j_{z,0} - v_z\rho_0\|\{|\overline{\f{v}}|^2\|\rho\|^2 + \|\f{j}\|^2\}
%$$
%and, subject to small enough initial datum, the bound above does not have a $C$-dependence, thus also trickling down to a better bound \eqref{lower_bound_B}. I will investigate this further}
Using the fact that $e_c(x)\gtrsim h^{-d}$ on cell $c$, we get
\begin{align}\label{T4}
T_4 \leq - \sum_{z=1}^{d}{\sum_{c \in \mathcal{C}_{h,d}}{\varphi(\|(\rho,\f{j})\|_{\nabla,c})h^{-d}\{\|\nabla j_z\|_c^2 + \overline{v}^2_z\|\nabla\rho\|_c^2\}}}
\end{align}
Combining \eqref{T1}--\eqref{T13}--\eqref{T2}--\eqref{T4} gives 
\begin{align}\label{partial_t_eta_multid}
\frac{1}{2}\partial_t \eta & \lesssim \tilde{C}(d)C_P^2( C_a |\overline{\f{v}}|^2 + |\overline{\f{v}}|^3)  \\
& \quad 
+ \sum_{z=1}^{d}{\sum_{c \in \mathcal{C}_{h,d}}{\left[\tilde{C}(d)(C_P^2+1)(C_a + |\overline{\f{v}}|) -  \varphi(\|(\rho,\f{j})\|_{\nabla,c})h^{-d}\right]\{\|\nabla j_z\|_c^2 + \overline{v}^2_z\|\nabla\rho\|_c^2\}}}.
\end{align}
Now, we split the cells of $\mathcal{C}_{h,d}$ into 
\begin{align*}
\mathcal{C}_{good} := \left\{ c\in\mathcal{C}_{h,d}\colon \|\rho,\f{j}\|_{\nabla,c} \leq (V+1)h^d \right\},\qquad \mathcal{C}_{bad} := \left\{ c\in\mathcal{C}_{h,d}\colon \|\rho,\f{j}\|_{\nabla,c} > (V+1)h^d \right\}.
\end{align*}
Now assume that $\|\rho\|\geq \gamma$, where $\gamma>0$ is chosen to satisfy \eqref{properties_constant_A_1}, \eqref{properties_constant_A_3}. Note that, in particular, \eqref{properties_constant_A_3} gives us 
$
\gamma  < C_a/(4\theta\|g\|_{\infty}).
$ 
Then, using Poincar\'e implies 
\begin{align*}
|\overline{v}|^2\|\nabla\rho\|^2 + \sum_{z=1}^{d}{\|\nabla j_z\|^2} \geq |\overline{v}|^2\frac{\gamma^2 - 2C^2_P}{2C^2_P}.
\end{align*}

Thus we rely on \eqref{properties_constant_A_1} and we choose $V$ in \eqref{cutoff} such that 
\begin{align}\label{choice_K}
(V+1)^2 < \left(|\overline{v}|^2\frac{\gamma^2 - 2C^2_P}{2C^2_P}.\right)^{1/2}
\end{align}
Therefore, if $\|\rho\|>\gamma$, from \eqref{partial_t_eta_multid} and \eqref{cutoff} we get
\begin{align}\label{bound_partial_eta}
\frac{1}{2}\partial_t \eta & \leq \tilde{C}(d)C_P^2(C_a|\overline{\f{v}}|^2 + |\overline{\f{v}}|^3)
+ \sum_{c\in\mathcal{C}_{good}}{\tilde{C}(d)(C_P^2+1)(C_a + |\overline{\f{v}}|)(V+1)h^d} \nonumber\\
& \quad + \sum_{c\in\mathcal{C}_{bad}}{\left[\tilde{C}(d)(C_P^2+1)(C_a + |\overline{\f{v}}|) - Wh^{-d}\right](V+1)h^d} \leq 0,
\end{align}
where the last inequality is true provided that
\begin{align}\label{lower_bound_B}
Wh^{-d} > C_1h^{-d} + C_2,
\end{align}
where we have set 
\begin{align*}
C_1 & := \tilde{C}(d)\left[(C_P^2+1)(C_a + |\overline{v}|) + (V+1)^{-1}C_P^2(C_a|\overline{v}|^2 + |\overline{v}|^3)\right], \\
C_2 & := C_a\tilde{C}(d)(C_P^2+1)(1 + |\overline{\f{v}}|).
\end{align*}
We now show that $\|\rho\|$ can not exceed the threshold $C_a/(4\theta \|g\|_{\infty})$, thus giving us the flocking property globally in time. Assume that $\tau := \inf\{t>0: \|\rho(\cdot,t)\|> C_a/(4\theta \|g\|_{\infty})\} < \infty$. Using the continuity of $\|\rho(t)\|$, call $\tau - \delta$ the last time at which $\|\rho\|$ attains value $\gamma$. In particular, $\|\rho(\tau-\delta)\|=\gamma, \|\rho(\tau)\|=C_a/(4\theta \|g\|_{\infty})$, $\|\rho(z)\|\in [\gamma,C_a/(4\theta \|g\|_{\infty})]$ for $z\in(\tau-\delta,\tau)$. 

Using \eqref{bound_partial_eta}, and adding and subtracting suitable quantities, we get 
\begin{align}
\eta(\tau) & = |\overline{\f{v}}|^2\|\rho(\tau)\|^2 + \sum_{z=1}^{d}{\| j_z(\tau)\|^2} \leq \eta(\tau-\delta)  = |\overline{\f{v}}|^2\|\rho(\tau-\delta)\|^2 + \sum_{z=1}^{d}{\| j_z(\tau-\delta)\|^2} \nonumber\\
& \leq |\overline{\f{v}}|^2 \gamma^2 + 2|\overline{\f{v}}|^2\|\rho(\tau-\delta)\|^2 + 2\sum_{z=1}^{d}{\|j_z(\tau-\delta)-\overline{v}_z\rho(\tau-\delta)\|^2}\nonumber\\
& \leq 3|\overline{\f{v}}|^2 \gamma^2 + 2\|\f{j}_0-\overline{\f{v}}\rho_0\|^2 e^{-C_a(\tau - \delta)}, \label{use_flocking}
\end{align}
where in \eqref{use_flocking} we have used the flocking property as in Lemma \ref{flocking_property_multi_d}, thanks to definition of the stopping time $\tau$. Furthermore, reordering the \eqref{use_flocking} and using assumption \eqref{properties_constant_A_3} gives
\begin{align}
\|\rho(\tau)\| \leq \left(3 \gamma^2 + \frac{2\|\f{j}_0 - \overline{\f{v}}\rho_0\|^2}{|\overline{\f{v}}|^2}\right)^{1/2} < \frac{C_a}{4\theta\|g\|_{\infty}}
\end{align} 
and this contradicts the fact, that by the definition of $\tau$, we also have $\|\rho(\tau)\| = \frac{C_a}{4\theta\|g\|_{\infty}}$. Therefore it must be $\tau = \infty$, and so the flocking property is global in time.

%{\bfseries Acknowledgements:}\\

\end{document}